\documentclass[12pt,oneside,reqno]{amsart}

\usepackage{amssymb}
\usepackage[active]{srcltx}
\usepackage{a4wide}
\usepackage{graphicx,palatino,pifont, times, caption}
\usepackage{oldgerm}
\usepackage{amsmath,amsthm}
\usepackage{amssymb}
\usepackage{lineno}
\usepackage{amstext}
\everymath{\displaystyle}
\usepackage{cite}
\usepackage{epsfig}
\usepackage{enumerate}
\usepackage{color}
 \usepackage{setspace} 
\usepackage{amssymb}
\newtheorem{theorem}{Theorem}[section]
\newtheorem{lemma}[theorem]{Lemma}

\theoremstyle{definition}

\theoremstyle{remark}
\newtheorem{remark}[theorem]{Remark}
\newtheorem{note}[theorem]{Note}
\usepackage[linkcolor=blue, urlcolor=blue, citecolor=blue,
colorlinks, bookmarks]{hyperref}

\newcommand{\be}{\begin{equation}}
\newcommand{\ee}{\end{equation}}

\begin{document}

\title[]{Barnsley-Navascu\'es fractal operators on Banach spaces on the Sierpi\'nski gasket}



\author{Asheesh Kumar Yadav}
\address{Department of Mathematics, Aligarh Muslim University, Aligarh, Uttar Pradesh, India, 202002}
\email{asheeshyadav7057@gmail.com}
\author{Himanshu Kumar}
\address{Department of Applied Sciences, IIIT Allahabad, Prayagraj, India, 211015}
\email{rss2021006@iiita.ac.in}
\author{Saurabh Verma}
\address{Department of Applied Sciences, IIIT Allahabad, Prayagraj, India, 211015}
\email{saurabhverma@iiita.ac.in}
\author{Bilel Selmi}
\address{Analysis, Probability \& Fractals Laboratory: LR18ES17,       
Faculty of Sciences of Monastir,   
Department of Mathematics, University of Monastir, 5000-Monastir
Tunisia}
\email{bilel.selmi@fsm.rnu.tn, \;bilel.selmi@isetgb.rnu.tn}





 

\begin{abstract}
In this article, we define fractal operators motivated by the works of Barnsley and  Navascu\'es on various function spaces such as energy space, Lebesgue space, and oscillation space on the well-known fractal domain Sierpi\'nski gasket. We further explore the properties of these operators from the perspectives of operator and approximation theory.
\end{abstract}

\maketitle



\section{Introduction}\label{section 1} 
Barnsley's work with Iterated Function System (IFS) brought forth the concept of fractal interpolation functions (FIFs) \cite{MF1}. Since then, there has been a consistent increase in interest regarding the application of FIFs for approximating functions. Navascu\'es \cite{NV1, NV2} presented a general framework for defining non-smooth versions of classical approximants of FIFs, widely known as $\alpha$-fractal functions. In view of her great contribution \cite{NV1, NV2, NV3, NVV2, Mnew2} to fractal functions and applications, we name the operator based on $\alpha$-fractal functions as Barnsley-Navascués fractal operators.
Motivated by the seminal work of Barnsley \cite{MF1}, Celik et al. \cite{B} devised a construction for the FIFs on $SG$. Their approach involved the consideration of a constant scaling factor and the selection of a height function that adheres to the properties of a harmonic function on $SG$. Ruan \cite{Ruan3} made a significant contribution by introducing FIFs on the p.c.f. fractals,  following the pioneering work of Kigami \cite{Kigami}.  Moreover, they derived a rigorous condition that ensures the finiteness of energy for linear FIFs. 
Ri and Ruan \cite{Ruan4} discussed uniform FIFs on $SG$, a special class defined by Celik et al. \cite{B}. They presented a necessary and sufficient condition for the uniform FIFs on $SG$ to have finite energy. In addition, they explored the properties of these operators. In a recent investigation by Sahu and Priyadarshi \cite{Ruan4}, the box dimension of the graph of FIFs, harmonic functions, and functions with finite energy on $SG$ was thoroughly examined. They successfully determined both the upper and lower bounds for this box dimension. Most recently, Verma and Sahu \cite{VS1} introduced the concept of a bounded variation function on $SG$. They further calculated the precise fractal dimension of the graph corresponding to continuous functions with a bounded variation on the $SG$. Megala and Prasad \cite{MP1,MP2,MP3} studied the spectrum of self-affine measures, which were also related to analytical properties of Lebesgue spaces. Several researchers have explored FIF from various perspectives. Some recent work on Hausdorff and other fractal dimensions and set-valued fractal functions with applications can be seen in \cite{AV1,AV2,VAM,VP1,VAD}. Interested readers may refer to the following works and the references therein \cite{ABP, AJ1, AJ2, MF3, SS2, HP2, NVV2, PV1, RSY, D} for fractal functions and their applications. 

The Sierpi\'nski gasket has a rich structure for analysis, allowing the study of approximation theory, spectral theory, and function spaces directly on a fractal.
Given a continuous function on $SG$, the $\alpha$-fractal operator maps it to another function that agrees at specific points but presents executed fractal behavior. This operator depends on a scaling function and a base function and is defined using self-referential equations derived from IFS \cite{MF2, H}. Several researchers have studied the geometric, analytic, and dimensional properties of these $\alpha$-fractal functions \cite{ ASV1, SS3, JV1, RBS, Mnew2}. Recent developments focus on understanding the behavior of fractal operators on function spaces, such as researchers use oscillation spaces to measure local variation and connect them closely with smoothness and H\"older continuity \cite{AC1, SS3, VS1}. Energy spaces, used in the context of resistance forms and Laplacians on fractals \cite{RBS, Kigami, RS}. $L_q$ spaces are important for the approximation, integrability, and variational analysis \cite{ GVS, HP2, JV1, NV3}.

Priyadrashi, Verma and their group \cite{MV1, Mverma1, Mverma2, Mverma3, VM, VS1}, Chandra and Abbas \cite{SS2,SS3,SS}, and Gurubachan et al. \cite{GVS, GVS2, GVS3} have provided detailed estimates and investigated the stability of $\alpha$-fractal functions within various function spaces. Their studies demonstrate that, under suitable conditions on the scaling function, the fractal operator remains bounded and continuous while preserving important properties such as norm bounds and smoothness levels. For recent developments on fractals and their applications, the reader is referred to \cite{Bandt, Cazassa, RBS, ADBJ, DS1, HP1, J1, PR1, MNS1, BS1, BS2} and the references therein.   

The fractal operator is also compatible with many other types of function spaces. For example, several authors have explored its action on Lebesgue spaces and Sobolev-type spaces on $SG$ \cite{ ES, KBV12}. Researchers have addressed approximation results using $\alpha$-fractal bases in \cite{NV1, NV2, VM}, established connections with differential equations on $SG$ in \cite{RBS, RS}, and applied these ideas to data modeling and real-world systems in \cite{GP1, KBV11}. In addition to analytical behavior, the dimensional aspects of these operators and their outputs, especially the box, Hausdorff, and generalized fractal dimensions, have been widely studied \cite{RS1, RHS1,  ASV1, SS, Fal, Fraser, RSP, YSL, YLC}.

In the article, we study the properties of fractal operators on function spaces defined on $SG$, such as bounds for their norms, fixed points, and bounded below. Furthermore, we have discussed compactness and Fredholmness with their index on these function spaces.
\par

This paper is assembled as follows. Section 2 represents the construction of FIFs on $SG$ and their properties, and also introduces Navascu\'es's $\alpha$-fractal function version of this FIF on $SG$. Section 3 collects our main results.
   
 \section{Fractal interpolation function on the Sierpi\'nski gasket}
We start with a brief description of the main ideas and a basic introduction to SG. The reader can look at \cite{Kigami, RS, Mverma1, VS1} for more details. Let $V_0=\{p_1,p_2,p_3\}$ be the vertices of an equilateral triangle on $\mathbb{R}^2$. Define the maps $L_i(x)=\frac{1}{2}( x + p_i),$ for $i=1,2,3$. These three maps are contraction maps of the plane and form an IFS. The Sierpi\'nski gasket is the attractor of this IFS:
$$ 
SG= L_1(SG) \cup L_2(SG) \cup L_3(SG).
$$
For a fixed $N \in \mathbb{N}$, define the composition $L_i=L_{i_1} \circ L_{i_2}\circ \dots \circ L_{i_N}$ for any sequence $ i=(i_1,i_2, \dots, i_N) \in I^N:=\{1,2,3\}^N$. The set of points we get by applying these maps to the initial set $V_0$ gives the  $N$-th level vertex set $V_N$ of $SG$. Now, let $B: V_N \to \mathbb{R}$ be a given function. We find an IFS whose attractor is the graph of a continuous function $f$ on $SG$ such that $f|_{V_N}=B$.

For each $w \in I^N$, define the map $W_{w} :SG \times \mathbb{R} \to SG \times \mathbb{R}$ by 
$$ 
W_{w}(x,z)=\Big(L_w(x),F_{w}(x,z)\Big).
$$
where $F_{w}(x,z): SG \times \mathbb{R} \to \mathbb{R}$ is satisfying the following conditions: 
$$ 
| F_{w}(x,z_1)-F_{w}(x,z_2)| \le c_{w} ~|z_1 - z_2|,
 $$ and
$ F_{w}(p_j,B(p_j))=B(L_w(p_j))$ for every $ j \in I$ with $c:= \max\limits_{w \in I^N} c_{w} < 1$. 

Accordingly, we define 
$$
F_{w}(x,z)= \alpha_{w}(x)z+q_{w}(x),
$$
where $\alpha_{w} : SG \to \mathbb{R} $ and $q_{w}: SG \to \mathbb{R}$ are continuous functions. These are chosen so that 
$$
q_w(p_i)=B(L_w(p_i))-\alpha_w(p_i)B(p_i),
$$ 
for each $i\in I$, and we also require $\|\alpha\|_{\infty} :=\max\{\|\alpha_{w} \|_{\infty}: w \in I^N\}<1, \|q\|_{\infty} :=\max\{\|q_{w} \|_{\infty}: w \in I^N\}$. 

Let $K=SG\times \mathbb{R}$.  Then the collection  $\mathcal{I}:=\{K; W_{w}: w \in I^N\}$ forms an IFS on $K$.
   
\begin{theorem}[\cite{Mverma1}]
Let $N \in \mathbb{N}$ and $B:V_N \to \mathbb{R}$ be given. The IFS $\mathcal{I}=\{K; W_{w}:w \in I^N\}$ defined above produces a unique continuous function $g_*: SG \to \mathbb{R}$ such that ${{g_*}|}_{V_N}=B$ and  satisfies the following self-referential equation
\begin{align*}
g_*(x)=\alpha_w(L_w^{-1}(x))g_*(L_w^{-1}(x))+q_w(L_w^{-1}(x))\; \forall\; x \in L_w(SG), w \in I^N. 
\end{align*}
\end{theorem}
Navascu\'es \cite{NV1, NV2} invented the idea of $\alpha$-fractal functions, which are used in approximation theory and the construction of fractal Schauder bases in various function spaces. Following her work, for a given continuous function $f :SG \to \mathbb{R}$, if we consider the above functions $q_{w}: SG \to \mathbb{R}$ as 
$$q_w(x)=f(L_w(x))-\alpha_w(x)b(x),$$ where $b$ is a continuous function satisfying $b|_{V_0}=f|_{V_0}$. Then from the above theorem, there exists a unique continuous function $f^{\alpha}_{N,b}: SG \to \mathbb{R}$ such that $f^{\alpha}_{N,b}|_{V_N}=f|_{V_N}$ and  satisfies the following self-referential equation:
$$
f^{\alpha}_{N,b}(x)=f(x)+\alpha_w(L_w^{-1}(x))f^{\alpha}_{N,b}(L_w^{-1}(x))- \alpha_w(L_w^{-1}(x))b(L_w^{-1}(x))\; \forall\; x \in L_w(SG), w \in I^N.
$$
\section{Main Results}
\subsection{Oscillation Spaces}\label{Sec4}
First, we recall the oscillation spaces which were first introduced in \cite{ADBJ} and further, extended and studied in \cite{AC1, JV1, Mverma1, Ri2, VS1, VJN1}. These spaces help us to understand how a function behaves on smaller parts of the Sierpi\'nski Gasket. They measure how much a function varies or oscillates at each level of refinement. These spaces contain continuous functions, and we control their smoothness in a special way. Because of these properties, oscillation spaces are useful for analyzing functions with complex or irregular patterns.

Let $g: SG \rightarrow \mathbb{R}$ be a continuous function on $SG$. For any natural number $n$, define the total oscillation of order $n$ by 
$$
R(n,g)= \sum_{w \in \{1,2,3\}^n} R_g[L_w(SG)],
$$ 
where 
$$
R_g[L_w(SG)]=\sup\{|g(x_1)-g(x_2)|: x_1,x_2 \in L_w(SG)\}.
$$
Now, for $ 0< \beta \le \frac{\log 3}{\log2}$, define a new class of continuous functions on $SG$ as follows:
$$ 
\mathcal{C}^{\beta}(SG) := \left\{g:SG \rightarrow \mathbb{R}\bigg|~g \text{ ~is~ continuous ~and}~ \sup_{n \in \mathbb{N}} \frac{R(n,g)}{2^{n\left(\frac{\log3}{\log2}-\beta\right)}} < \infty\right\}.
$$
It is known (see \cite[Theorem 4]{Mverma1}) that the set $\mathcal{C}^{\beta}(SG)$, with the norm
$$
\|g\|_{\mathcal{C}^{\beta}}:= \|g\|_{\infty} + \sup_{n \in \mathbb{N} }\frac{ R(n,g)}{2^{n\left(\frac{\log3}{\log2}-\beta\right)}},
$$
is a Banach space.

    

For a given $f \in \mathcal{C}^{\beta}(SG)$.define a function $q_{w}: SG \to \mathbb{R}$ as 
$$
q_w(x)=f(L_w(x))-\alpha_w(x)b(x) \quad \text{ for all }  w \in I^N,
$$ 
where $b \in \mathcal{C}^{\beta}(SG)$ is a base function that agrees with $f$on the set $V_0$, i.e., $b|_{V_0}=f|_{V_0}$. Using item $3$ of \cite[Lemma 1]{Mverma1}, this function $q_w$ also belongs to $\mathcal{C}^{\beta}(SG)$. Now, suppose the following condition holds: 
$$
\max \Bigg\{\| \alpha \|_{\infty} +\frac{3^n}{2^{n\left(\frac{\log3}{\log2}-\beta\right)}}\sup_{w \in I^n} \sup_{m \in \mathbb{N}}\frac{R(m,\alpha_{w})}{2^{m\left(\frac{\log3}{\log2}-\beta\right)}},\frac{3^n \|\alpha\|_{\infty}}{2^{n\left(\frac{\log3}{\log2}-\beta\right)}} \Bigg\}<1.
$$ 
Then, by \cite[Theorem 5]{Mverma1}, there exists a unique function $f^{\alpha}_{N,b} \in \mathcal{C}^{\beta}(SG)$ such that:
\begin{itemize}
\item $f^{\alpha}_{N,b}$ agrees with $f$ on the vertex set $V_N$, i.e., $f^{\alpha}_{N,b}|_{V_N}=f|_{V_N}$, and 
\item it satisfies the self-referential equation:
$$
f^{\alpha}_{N,b}(x)=f(x)+\alpha_w(L_w^{-1}(x))f^{\alpha}_{N,b}(L_w^{-1}(x))- \alpha_w(L_w^{-1}(x))b(L_w^{-1}(x))\; 
$$
\end{itemize}
for all $x \in L_w(SG)$, and for all $w \in I^N$.

Now, let $T:\mathcal{C}^{\beta}(SG) \to \mathcal{C}^{\beta}(SG)$ be a bounded linear operator such that $(Tg)|_{V_0}=g|_{V_0}$ for any $g \in \mathcal{C}^{\beta}(SG)$. if we consider the base function $b=Tf$, then we can define a map $\mathcal{F}^{\alpha}_{N,T}: \mathcal{C}^{\beta}(SG) \to \mathcal{C}^{\beta}(SG)$ by 
$$
\mathcal{F}^{\alpha}_{N,T}(f)=f^{\alpha}_{N,b}=f^{\alpha}_{N,T}.
$$
This map $\mathcal{F}^{\alpha}_{N,T}$ is a bounded linear operator. Following the terminology of Navascu\'es, we call this operator a fractal operator on $\mathcal{C}^{\beta}(SG)$.

Now, let us explore the various properties of this operator.

\begin{lemma} \cite[Lemma 1]{Cazassa}\label{Lemma1} Let $(X,\|.\|)$ be a Banach space, $T: X \rightarrow X$ be a linear operator. Suppose there exist constants $\lambda_1, \lambda_2 \in[0,1)$ such that
$$
\|T x-x\| \leq \lambda_1\|x\|+\lambda_2\|T x\|, \quad \forall x \in X .
$$
Then $T$ is a topological isomorphism, and
$$
\frac{1-\lambda_2}{1+\lambda_1}\|x\| \leq\left\|T^{-1} x\right\| \leq \frac{1+\lambda_2}{1-\lambda_1}\|x\|, \quad \forall x \in X .
$$
\end{lemma}
\begin{theorem}\label{thm32}
    Let $Id$ be the identity operator on $\mathcal{C}^{\beta}(S G)$. Denote $$\|\alpha\|_{\infty}=\max \left\{\left\|\alpha_w\right\|_{\infty}: w\in I^N\right\}~\text{and}~\|\alpha\|_{\mathcal{C}^{\beta}}=\max \left\{\left\|\alpha_w\right\|_{\mathcal{C}^{\beta}}: w\in I^N\right\}.$$ If $\left\|\alpha\right\|_{\infty}+\left\|\alpha\right\|_{\mathcal{C}^{\beta}}<\frac{1}{2^{N\beta}}$, then the following statements hold:
    \begin{enumerate}
        \item\label{item1}  Let $f \in \mathcal{C}^{\beta}(S G)$ be an arbitrary. Then the perturbation error is of the form:
$$
\left\|f^\alpha_{N,T}-f\right\|_{\mathcal{C}^{\beta}} \leq \frac{2^{N\beta}(\|\alpha\|_{\infty}+\left\|\alpha\right\|_{\mathcal{C}^{\beta}})}{1- 2^{N\beta}(\left\|\alpha\right\|_{\infty}+\left\|\alpha\right\|_{\mathcal{C}^{\beta}})}\|f-T f\|_{\mathcal{C}^{\beta}} .
$$
Note that, for $\|\alpha\|_{\mathcal{C}^{\beta}}=0$ or $T=I$, we have $\mathcal{F}^\alpha_{N,T}= Id.$
\item Under the uniform norm on $\mathcal{C}^{\beta}(S G)$, the fractal operator $\mathcal{F}^\alpha_{N,T}$ is a bounded linear operator. Moreover, the operator norm holds
$$
\left\|\mathcal{F}^\alpha_{N,T}\right\| \leq 1+\frac{2^{N\beta}(\|\alpha\|_{\infty}+\left\|\alpha\right\|_{\mathcal{C}^{\beta}})\|I d-T\|}{1-2^{N\beta}(\left\|\alpha\right\|_{\infty}+\left\|\alpha\right\|_{\mathcal{C}^{\beta}})} .
$$
\item If $(\left\|\alpha\right\|_{\infty}+\left\|\alpha\right\|_{\mathcal{C}^{\beta}})\|T\|<\frac{1}{2^{N\beta}}$, then $\mathcal{F}^\alpha_{N,T}$ is bounded below. In particular, $\mathcal{F}^\alpha_{N,T}$ is injective.
\item\label{item4} If $(\left\|\alpha\right\|_{\infty}+\left\|\alpha\right\|_{\mathcal{C}^{\beta}})\|T\|<\frac{1}{2^{N\beta}}$, then the inverse of $\mathcal{F}^\alpha_{N,T}$ exists which is also bounded and therefore a topological automorphism. Furthermore,
$$
\frac{1-(\left\|\alpha\right\|_{\infty}+\left\|\alpha\right\|_{\mathcal{C}^{\beta}})}{1+(\left\|\alpha\right\|_{\infty}+\left\|\alpha\right\|_{\mathcal{C}^{\beta}})\|T\|}\leq\left\|\left(\mathcal{F}^\alpha_{N,T}\right)^{-1}\right\| \leq \frac{1+(\left\|\alpha\right\|_{\infty}+\left\|\alpha\right\|_{\mathcal{C}^{\beta}})}{1-(\left\|\alpha\right\|_{\infty}+\left\|\alpha\right\|_{\mathcal{C}^{\beta}})\|T\|}.
$$
\item\label{item5} If $\|\alpha\|_{\mathcal{C}^{\beta}} \neq 0$ then the fixed points of $T$ are also the the fixed points of $\mathcal{F}^\alpha_{N,T}$. If at least one $\alpha_{\omega}$ is nowhere zero, then the fixed points of $\mathcal{F}^\alpha_{N,T}$ are also the fixed points of $T$.
\item\label{item6} If the point spectrum of $T$ contains $1$, then $\left\|\mathcal{F}^\alpha_{N,T}\right\|_{ \mathcal{C}^{\beta}}\geq 1$.
\item The fractal operator $\mathcal{F}^\alpha_{N,T}$ is not the compact operator in the case of $(\left\|\alpha\right\|_{\infty}+\left\|\alpha\right\|_{\mathcal{C}^{\beta}})\|T\|<\frac{1}{2^{N\beta}}$.
\item If $(\left\|\alpha\right\|_{\infty}+\left\|\alpha\right\|_{\mathcal{C}^{\beta}})\|T\|<\frac{1}{2^{N\beta}}$, then $\mathcal{F}^\alpha_{N,T}$ is Fredholm and its index is $0.$
  \end{enumerate}
\end{theorem}
\begin{proof}\noindent
\begin{enumerate}
\item  It follows from the self-referential equation that 
$$
f^{\alpha}_{N,b}(x)-f(x)=\alpha_w(L_w^{-1}(x))f^{\alpha}_{N,b}(L_w^{-1}(x))- \alpha_w(L_w^{-1}(x))b(L_w^{-1}(x))\; \forall\; x \in L_w(SG), w \in I^N.
$$
We have 
$$
\left\|f^\alpha_{N,T}-f\right\|_{\mathcal{C}^{\beta}}= \|f^\alpha_{N,T}-f\|_{\infty} + \sup_{n \in \mathbb{N} }\frac{ R(n,f^\alpha_{N,T}-f)}{2^{n\left(\frac{\log3}{\log2}-\beta\right)}}.
$$
and 
\begin{eqnarray}
\left|f^\alpha_{N,T}(x)-f(x)\right|
&=&\left|\alpha_w(L_w^{-1}(x))f^{\alpha}_{N,T}(L_w^{-1}(x))- \alpha_w(L_w^{-1}(x))b(L_w^{-1}(x))\right|\nonumber\\
&=& \left|\alpha_w(L_w^{-1}(x))\right|\left|f^{\alpha}_{N,T}(L_w^{-1}(x))-b(L_w^{-1}(x))\right|\nonumber\\
&\leq& \left\|\alpha\right\|_{\infty}\left\|f^{\alpha}_{N,T}-b\right\|_{\infty}.\label{eq0}
\end{eqnarray}
Which implies that 
\begin{eqnarray}
 \left\|f^\alpha_{N,T}-f\right\|_{\infty}&\leq&\left\|\alpha\right\|_{\infty}\left\|f^{\alpha}_{N,T}-b\right\|_{\infty}\\&\leq&\left\|\alpha\right\|_{\infty}\left\|f^{\alpha}_{N,T}-f\right\|_{\infty}+\left\|\alpha\right\|_{\infty}\left\|f-b\right\|_{\infty}.
\end{eqnarray}
Now, we can deduce that 
\begin{eqnarray}\label{eq1}
    \left\|f^\alpha_{N,T}-f\right\|_{\infty}\leq\frac{\left\|\alpha\right\|_{\infty}}{1-\left\|\alpha\right\|_{\infty}}\left\|f-T f\right\|_{\infty}.
\end{eqnarray}
For $n>N$, we have
\begin{eqnarray*}
&& R(n,f^\alpha_{N,T}-f)\\&=&\sum_{\sigma \in \{1,2,3\}^n}\sup\Big\{|(f^\alpha_{N,T}-f)(L_{\sigma}(x_1))-(f^\alpha_{N,T}-f)(L_{\sigma}(x_2))|: x_1,x_2 \in SG\Big\}\\&=&\sum_{w \in \{1,2,3\}^N, \tau\in \{1,2,3\}^{n-N}}\sup\Big\{|(f^\alpha_{N,T}-f)(L_{w}(L_{\tau}(x_1)))-(f^\alpha_{N,T}-f)(L_{w}(L_{\tau}(x_2)))|: x_1,x_2 \in SG\Big\}\\&=&\sum_{w \in \{1,2,3\}^N, \tau\in \{1,2,3\}^{n-N}}\sup_{x_1,x_2 \in SG}\Big\{|\alpha_w(L_{\tau}(x_1))((f^{\alpha}_{N,T}-b)(L_{\tau}(x_1))-\alpha_w(L_{\tau}(x_2))((f^{\alpha}_{N,T}-b)(L_{\tau}(x_2))|\Big\}\\
 &\leq& 3^{N}\left\|\alpha\right\|_{\infty}\sum_{\tau  \in \{1,2,3\}^{n-N}}\sup\Big\{|((f^{\alpha}_{N,T}-b)(L_{\tau}(x_1))-((f^{\alpha}_{N,T}-b)(L_{\tau}(x_2))|: x_1,x_2 \in SG\Big\}\\&+&\left\|f^{\alpha}_{N,T}-b\right\|_{\infty}\sum_{w \in \{1,2,3\}^N, \tau\in \{1,2,3\}^{n-N}}\sup\Big\{|(\alpha_w(L_{\tau}(x_1))-(\alpha_w(L_{\tau}(x_2))|: x_1,x_2 \in SG\Big\}\\
 &\leq& 3^{N}\left\|\alpha\right\|_{\infty} \Biggl(\sum_{\tau \in \{1,2,3\}^{n-N}}\sup\Big\{|((f^{\alpha}_{N,T}-f)(L_{\tau}(x_1))-((f^{\alpha}_{N,T}-f)(L_{\tau}(x_2))|: x_1,x_2 \in SG\Big\}\\&+&\sum_{\tau \in \{1,2,3\}^{n-N}}\sup\Big\{|((f-b)(L_{\tau}(x_1))-((f-b)(L_{\tau}(x_2))|: x_1,x_2 \in SG\Big\}\Biggl)\\&+&\left\|f^{\alpha}_{N,T}-b\right\|_{\infty}\sum_{w\in \{1,2,3\}^{N}}R(n-N,\alpha_w)\\
 &\leq& 3^{N}\left\|\alpha\right\|_{\infty}\Biggl(R(n-N,f^\alpha_{N,T}-f)+R(n-N,f-b)\Biggl)+\left\|f^{\alpha}_{N,T}-b\right\|_{\infty}\sum_{w\in \{1,2,3\}^{N}}R(n-N,\alpha_w).
 \end{eqnarray*}
 This implies that 
 \begin{align*}
  &\frac{R(n,f^\alpha_{N,T}-f)}{2^{n\left(\frac{\log3}{\log2}-\beta\right)}}\\& \leq 3^{N}\left\|\alpha\right\|_{\infty}\frac{R(n-N,f^\alpha_{N,T}-f)+R(n-N,f-b)}{2^{n\left(\frac{\log3}{\log2}-\beta\right)}}+\left\|f^{\alpha}_{N,T}-b\right\|_{\infty}\sum_{w\in \{1,2,3\}^{N}} \frac{R(n-N,\alpha_w)}{2^{n\left(\frac{\log3}{\log2}-\beta\right)}}\\&=\frac{3^{N}\left\|\alpha\right\|_{\infty}}{2^{N\left(\frac{\log3}{\log2}-\beta\right)}}\frac{R(n-N,f^\alpha_{N,T}-f)+R(n-N,f-b)}{2^{(n-N)\left(\frac{\log3}{\log2}-\beta\right)}}+\frac{\left\|f^{\alpha}_{N,T}-b\right\|_{\infty}}{2^{N\left(\frac{\log3}{\log2}-\beta\right)}}\sum_{w\in \{1,2,3\}^{N}} \frac{R(n-N,\alpha_w)}{2^{(n-N)\left(\frac{\log3}{\log2}-\beta\right)}}
 \end{align*}
Thus, we have  
\begin{eqnarray*}\label{eq2}
    \sup_{n \in \mathbb{N} }\frac{ R(n,f^\alpha_{N,T}-f)}{2^{n\left(\frac{\log3}{\log2}-\beta\right)}}\leq \frac{3^{N}\left\|\alpha\right\|_{\infty}}{\mathcal{A}} \sup_{n \in \mathbb{N} }\frac{ R(n,f-T f)}{2^{n\left(\frac{\log3}{\log2}-\beta\right)}}+\frac{\left\|f^{\alpha}_{N,T}-T f\right\|_{\infty}}{\mathcal{A}}\sum_{w\in \{1,2,3\}^{N}}\sup_{n \in \mathbb{N} }\frac{ R(n,\alpha_w)}{2^{n\left(\frac{\log3}{\log2}-\beta\right)}},
\end{eqnarray*}
where $\mathcal{A}=2^{N\left(\frac{\log3}{\log2}-\beta\right)}.$
It follows that
\begin{align*}
   & \left\|f^\alpha_{N,T}-f\right\|_{\mathcal{C}^{\beta}} \\&\leq \max\left\{\frac{\left\|\alpha\right\|_{\infty}}{1-\left\|\alpha\right\|_{\infty}}, \frac{3^{N}\left\|\alpha\right\|_{\infty}}{\mathcal{A}}\right\} \left(\left\|f-T f\right\|_{\infty}+\sup_{n \in \mathbb{N} }\frac{ R(n,f-T f)}{2^{n\left(\frac{\log3}{\log2}-\beta\right)}}\right)\\&+\frac{\left\|f^{\alpha}_{N,T}-T f\right\|_{\infty}}{\mathcal{A}}\sum_{w\in \{1,2,3\}^{N}}\sup_{n \in \mathbb{N} }\frac{ R(n,\alpha_w)}{2^{n\left(\frac{\log3}{\log2}-\beta\right)}}\\
    &\leq  \max\left\{\frac{\left\|\alpha\right\|_{\infty}}{1-\left\|\alpha\right\|_{\infty}}, \frac{2^{N\beta}\left\|\alpha\right\|_{\infty}}{1-2^{N\beta}\left\|\alpha\right\|_{\infty}}\right\}\|f-T f\|_{\mathcal{C}^{\beta}}+\frac{2^{N\beta}\|\alpha\|_{\mathcal{C}^{\beta}}}{1-2^{N\beta}\left\|\alpha\right\|_{\infty}} \left\|f^{\alpha}_{N,T}-T f\right\|_{\mathcal{C}^{\beta}}\\&=  \frac{2^{N\beta}\left\|\alpha\right\|_{\infty}}{1-2^{N\beta}\left\|\alpha\right\|_{\infty}}\|f-T f\|_{\mathcal{C}^{\beta}}+\frac{2^{N\beta}\|\alpha\|_{\mathcal{C}^{\beta}}}{1-2^{N\beta}\left\|\alpha\right\|_{\infty}} \left\|f^{\alpha}_{N,T}-T f\right\|_{\mathcal{C}^{\beta}}.
\end{align*}
This implies that 
\begin{align*}
   \left\|f^\alpha_{N,T}-f\right\|_{\mathcal{C}^{\beta}}\leq  \frac{2^{N\beta}(\left\|\alpha\right\|_{\infty}+\left\|\alpha\right\|_{\mathcal{C}^{\beta}})}{1- 2^{N\beta}( \left\|\alpha\right\|_{\infty}+\left\|\alpha\right\|_{\mathcal{C}^{\beta}})}\|f-T f\|_{\mathcal{C}^{\beta}}
\end{align*}
\item By using the first assertion, we have
$$
\left\|f^\alpha_{N,T}\right\|_{\mathcal{C}^{\beta}}-\left\|f\right\|_{\mathcal{C}^{\beta}}\leq\left\|f^\alpha_{N,T}-f\right\|_{\mathcal{C}^{\beta}}\leq\frac{2^{N\beta}(\left\|\alpha\right\|_{\infty}+\left\|\alpha\right\|_{\mathcal{C}^{\beta}})}{1- 2^{N\beta}( \left\|\alpha\right\|_{\infty}+\left\|\alpha\right\|_{\mathcal{C}^{\beta}})}\|f-T f\|_{\mathcal{C}^{\beta}}.
$$
or
$$
\left\|\mathcal{F}^{\alpha}_{N,T}(f)\right\|_{\mathcal{C}^{\beta}}-\left\|f\right\|_{\mathcal{C}^{\beta}}\leq\left\|f^\alpha_{N,T}-f\right\|_{\mathcal{C}^{\beta}}\leq\frac{2^{N\beta}(\left\|\alpha\right\|_{\infty}+\left\|\alpha\right\|_{\mathcal{C}^{\beta}})}{1- 2^{N\beta}( \left\|\alpha\right\|_{\infty}+\left\|\alpha\right\|_{\mathcal{C}^{\beta}})}\left\|f-T f\right\|_{\mathcal{C}^{\beta}}.
$$
So, 
$$
\left\|\mathcal{F}^{\alpha}_{N,T}(f)\right\|_{\mathcal{C}^{\beta}}\leq\left[1+\frac{2^{N\beta}(\left\|\alpha\right\|_{\infty}+\left\|\alpha\right\|_{\mathcal{C}^{\beta}})\left\|Id-T \right\|}{1-2^{N\beta}(\left\|\alpha\right\|_{\infty}+\left\|\alpha\right\|_{\mathcal{C}^{\beta}})}\right]\left\|f \right\|_{\mathcal{C}^{\beta}}.
$$
Which implies that 
$$
\left\|\mathcal{F}^{\alpha}_{N,T}\right\|\leq1+\frac{2^{N\beta}(\left\|\alpha\right\|_{\infty}+\left\|\alpha\right\|_{\mathcal{C}^{\beta}})\left\|Id-T \right\|}{1-2^{N\beta}(\left\|\alpha\right\|_{\infty}+\left\|\alpha\right\|_{\mathcal{C}^{\beta}})}.
$$
\item  By performing the similar calculation in part $(1)$, for $n>N$ we get
\begin{eqnarray*}\label{eq3}
    \sup_{n \in \mathbb{N} }\frac{ R(n,f^\alpha_{N,T}-f)}{2^{n\left(\frac{\log3}{\log2}-\beta\right)}}\leq \frac{3^{N}\left\|\alpha\right\|_{\infty}}{2^{N\left(\frac{\log3}{\log2}-\beta\right)}} \sup_{n \in \mathbb{N} }\frac{ R(n,f^\alpha_{N,T}-T f)}{2^{n\left(\frac{\log3}{\log2}-\beta\right)}}+\frac{\left\|f^{\alpha}_{N,T}-T f\right\|_{\infty}}{2^{N\left(\frac{\log3}{\log2}-\beta\right)}}\sum_{w\in \{1,2,3\}^{N}}\sup_{n \in \mathbb{N} }\frac{ R(n,\alpha_w)}{2^{n\left(\frac{\log3}{\log2}-\beta\right)}}.
\end{eqnarray*}
Thus, from the above inequality and from \eqref{eq0}, we obtain
\begin{eqnarray*}
    \left\|f\right\|_{\mathcal{C}^{\beta}}-\left\|f^\alpha_{N,T}\right\|_{\mathcal{C}^{\beta}}&\leq&\left\|f^\alpha_{N,T}-f\right\|_{\mathcal{C}^{\beta}}\\&\leq& 2^{N\beta}(\left\|\alpha\right\|_{\infty}+\left\|\alpha\right\|_{\mathcal{C}^{\beta}})\left\|f^\alpha_{N,T}-T f\right\|_{\mathcal{C}^{\beta}}\\
    &\leq& 2^{N\beta}(\left\|\alpha\right\|_{\infty}+\left\|\alpha\right\|_{\mathcal{C}^{\beta}})\left(\left\|f^\alpha_{N,T}\right\|_{\mathcal{C}^{\beta}}+\left\|T\right\|\left\|f\right\|_{\mathcal{C}^{\beta}}\right).
\end{eqnarray*}
Which implies that 
$$
\Big(1-2^{N\beta}(\left\|\alpha\right\|_{\infty}+\left\|\alpha\right\|_{\mathcal{C}^{\beta}})\left\|T\right\|\Big)\left\|f\right\|_{\mathcal{C}^{\beta}}\leq \Big(1+2^{N\beta}(\left\|\alpha\right\|_{\infty}+\left\|\alpha\right\|_{\mathcal{C}^{\beta}})\Big)\left\|f^\alpha_{N,T}\right\|_{\mathcal{C}^{\beta}}
$$
By using $(\left\|\alpha\right\|_{\infty}+\left\|\alpha\right\|_{\mathcal{C}^{\beta}})\|T\|<\frac{1}{2^{N\beta}},$ we get
$$
\left\|f\right\|_{\mathcal{C}^{\beta}}\leq \frac{1+2^{N\beta}(\left\|\alpha\right\|_{\infty}+\left\|\alpha\right\|_{\mathcal{C}^{\beta}})}{1-2^{N\beta}(\left\|\alpha\right\|_{\infty}+\left\|\alpha\right\|_{\mathcal{C}^{\beta}})\left\|T\right\|}\left\|f^\alpha_{N,T}\right\|_{\mathcal{C}^{\beta}}.
$$
This implies that $\mathcal{F}^\alpha_{N,T}$ is bounded below.
\item  We have
\begin{eqnarray*}
    \left\|f-\mathcal{F}^{\alpha}_{N,T}(f)\right\|_{\mathcal{C}^{\beta}}&\leq& 2^{N\beta}(\left\|\alpha\right\|_{\infty}+\left\|\alpha\right\|_{\mathcal{C}^{\beta}})\left\|\mathcal{F}^{\alpha}_{N,T}(f)-T f\right\|_{\mathcal{C}^{\beta}}\\&\leq&
    2^{N\beta}{(\left\|\alpha\right\|_{\infty}+\left\|\alpha\right\|_{\mathcal{C}^{\beta}})}\left\|\mathcal{F}^{\alpha}_{N,T}(f) \right\|_{\mathcal{C}^{\beta}}+2^{N\beta}(\left\|\alpha\right\|_{\infty}+\left\|\alpha\right\|_{\mathcal{C}^{\beta}})\left\|T\right\|\left\|f \right\|_{\mathcal{C}^{\beta}}.
\end{eqnarray*}
It follows from  $2^{N\beta}(\left\|\alpha\right\|_{\infty}+\left\|\alpha\right\|_{\mathcal{C}^{\beta}})\|T\|<1$  and Lemma \ref{Lemma1} that the fractal operator $\mathcal{F}^{\alpha}_{N,T}$
is a topological isomorphism and in particular we have 
$$
\frac{1-2^{N\beta}(\left\|\alpha\right\|_{\infty}+\left\|\alpha\right\|_{\mathcal{C}^{\beta}})}{1+2^{N\beta}(\left\|\alpha\right\|_{\infty}+\left\|\alpha\right\|_{\mathcal{C}^{\beta}})\|T\|}\leq\left\|\left(\mathcal{F}^\alpha_{N,T}\right)^{-1}\right\| \leq \frac{1+2^{N\beta}(\left\|\alpha\right\|_{\infty}+\left\|\alpha\right\|_{\mathcal{C}^{\beta}})}{1-2^{N\beta}(\left\|\alpha\right\|_{\infty}+\left\|\alpha\right\|_{\mathcal{C}^{\beta}})\|T\|}.
$$

\item By item \eqref{item1}, we deduce
$$
\left\|\mathcal{F}^{\alpha}_{N,T}(f)-f\right\|_{\mathcal{C}^{\beta}} \leq \frac{2^{N\beta}(\|\alpha\|_{\infty}+\left\|\alpha\right\|_{\mathcal{C}^{\beta}})}{1-2^{N\beta}(\left\|\alpha\right\|_{\infty}+\left\|\alpha\right\|_{\mathcal{C}^{\beta}})}\|f-T f\|_{\mathcal{C}^{\beta}} .
$$
For any fixed point $g^*$ of $T$, we get 
$$
\left\|\mathcal{F}^{\alpha}_{N,T}(g^*)-g^*\right\|_{\mathcal{C}^{\beta}} \leq 0,
$$
which yields $\mathcal{F}^{\alpha}_{N,T}(g^*)=g^*.$ Now for the other part, we know that $$\mathcal{F}^{\alpha}_{N,T}(f)(x)
  =f(x)+\alpha_w(L_w^{-1}(x))\mathcal{F}^{\alpha}_{N,T}(f)(L_w^{-1}(x))- \alpha_w(L_w^{-1}(x))b(L_w^{-1}(x)),$$
  for $\mathcal{F}^{\alpha}_{N,T}(h^*)
  =h^*$, we get
  $$\alpha_w(L_w^{-1}(x))(h^*- Th^*)(L_w^{-1}(x))=0~\text{for}~x\in L_{\omega}(SG),~\omega\in\{1,2,3\}^N.$$
Since $\alpha_{\omega}$ (by assumption at least one $\alpha_w$) is nowhere zero, we obtain
$$(h^*- Th^*)(L_w^{-1}(x))=0~\text{for}~x\in L_{\omega}(SG).$$
Equivalently,
$$(h^*- Th^*)(x)=0~\text{for}~x\in SG.$$
Hence $Th^*=h^*$, completes the proof.

\item Take $f_0\in\mathcal{C}^{\beta}(SG)$ such that $Tf_0=f_0~\text{and }\|f_0\|_{\mathcal{C}^{\beta}}=1$. Now by item \eqref{item5} of the following theorem $\mathcal{F}^{\alpha}_{N,T}(f_0)=f_0.$ Thus by the definition of the operator norm we get 
\begin{eqnarray*}
    \|\mathcal{F}^{\alpha}_{N,T}\|&\geq& \frac{\|\mathcal{F}^{\alpha}_{N,T}(f_0)\|_{\mathcal{C}^{\beta}}}{\|f_0\|_{\mathcal{C}^{\beta}}}=1.
\end{eqnarray*}

\item By \eqref{item4}, we know that $(\mathcal{F}^{\alpha}_{N,T})^{-1}$ exists and also a bounded linear operator. Suppose that $\mathcal{F}^{\alpha}_{N,T}$ is a compact operator. Thus, $I=\mathcal{F}^{\alpha}_{N,T}\circ(\mathcal{F}^{\alpha}_{N,T})^{-1}:\mathcal{C}^{\beta}(SG) \to \mathcal{C}^{\beta}(SG)$ is also a compact operator. But this contradicts the fact that $\mathcal{C}^{\beta}(SG)$ is an infinite dimensional space. Hence, $\mathcal{F}^{\alpha}_{N,T}$ is not a compact operator. 

\item For an invertible operator $T:X\to Y$, $T^{*}$ is also an invertible operator. Thus by combining it with \eqref{item4}, $({\mathcal{F}^{\alpha}_{N,T}})^{*}$ is also invertible. Hence, $\mathcal{F}^{\alpha}_{N,T}$ is Fredholm. Moreover, by \eqref{item4} we can conclude that 
$$\text{index}(\mathcal{F}^{\alpha}_{N,T}))=\dim(\text{Ker}(\mathcal{F}^{\alpha}_{N,T}))-\dim(\text{Ker}((\mathcal{F}^{\alpha}_{N,T})^*))=0.$$

\end{enumerate}
\end{proof}
\subsection{Energy space on SG}\label{Sec5}

Start with the vertex set $V_0$, and define $ \Lambda_0$ as the complete graph on $V_0$ (every pair of vertices is connected). Suppose we have construct the graph $\Lambda_{n-1}$ with vertex set $V_{n-1}$ for some $n \geq 1$. Then, we define the graph $\Lambda_n$ on $V_n$ as follows: for any $x,y \in V_n,$ we say $ x \sim_n y$ if and only if $x=L_i(x_1),y=L_i(y_1)$ with $x_1 \sim_{n-1} y_1$ and $i \in I$. In other words, $x \sim_n y$ if and only if there exists $i \in I^n$ such that $x,y \in L_i(V_0)$. For each level $n=0,1,2, \dots,$ define the graph energy of a function $f$ on $\mathcal{G}_n$ as
$$
\mathcal{E}_n(f):= \Big(\frac{5}{3} \Big)^n \sum_{x \sim_n y}(f(x)-f(y))^2.
$$
This energy measures how ``smooth" $f$ is on the graph. The energy at level $n-1$ is the smallest possible energy among all extensions of $f$ from $V_{n-1}$ to $V_n$, i.e.,
$$ 
\mathcal{E}_{n-1}(f)= \min \{ \mathcal{E}_n(\tilde{f}) : \tilde{f}|_{V_{n-1}}=f \}.
$$ 
As $n$ increases, the sequence $\{\mathcal{E}_{n}(f)\}$ is non-decreasing. The energy of $f$ on the entire vertex set $V_* :=\cup_{m=0}^{\infty}V_m$ is defined as
$$ 
\mathcal{E}(f):= \lim_{n \rightarrow \infty}\mathcal{E}_n(f).
$$ 
if $f \in \mathcal{C}(SG)$ and $\mathcal{E}(f) < \infty$, then we called $f$ has finite energy. Such functions are uniformly continuous on  $V_*$, and since $V_*$ is dense in $SG$, they extend uniquely to continuous functions on $SG$.

If a continuous function $f$ satisfies $ \mathcal{E}_{n-1}(f)=\mathcal{E}_n(f)$ for all $n \geq 1,$ then we call $f$ is a harmonic function on $SG$. 

Finally, define the energy space as
$$
\text{dom}(\mathcal{E})=\{g \in \mathcal{C}(SG):  \mathcal{E}(g)< \infty\}.
$$
From the reference \cite[Theorem $1.4.2$]{RS}, the set $\text{dom}(\mathcal{E})$, equipped with the norm  
$$ 
\|g\|_{\mathcal{E}}:=\|g\|_{\infty} + \sqrt{\mathcal{E}(g)},
$$
is a Banach space.
     
\par
Let $f \in \text{dom}(\mathcal{E})$ be a given function. Define a function $q_{w}: SG \to \mathbb{R}$ as 
$$
q_w(x)=f(L_w(x))-\alpha_w(x)b(x),
$$
where $b \in \text{dom}(\mathcal{E})$ is another function such that $b$ and $f$ agrees on vertex set $V_0$, i.e., $b|_{V_0}=f|_{V_0}$.
If we assume that $\|\alpha\|_{\mathcal{E}} < \frac{1}{2\sqrt{5^N }},$ then from \cite[Theorem 8]{Mverma1}, there exists a unique function $f^{\alpha}_{N,b} \in \text{dom}(\mathcal{E})$ such that:
\begin{itemize}
\item $f^{\alpha}_{N,b}|_{V_N}=f|_{V_N},$ and 
\item it satisfies the self-referential equation:
$$
f^{\alpha}_{N,b}(x)=f(x)+\alpha_w(L_w^{-1}(x))f^{\alpha}_{N,b}(L_w^{-1}(x))- \alpha_w(L_w^{-1}(x))b(L_w^{-1}(x)), 
$$
\end{itemize}
for all $x \in L_w(SG)$, and for all $w \in I^N$.

Now, a take bounded linear operator $T:\text{dom}(\mathcal{E}) \to \text{dom}(\mathcal{E}) $ such that $(Tg)|_{V_0}=g|_{V_0}$ for any function $g$. If we choose the base function $b=Tf$, then we define a map $\mathcal{F}^{\alpha}_{N,T}: \text{dom}(\mathcal{E}) \to \text{dom}(\mathcal{E})$ defined by
$$
\mathcal{F}^{\alpha}_{N,T}(f)=f^{\alpha}_{N,b}=f^{\alpha}_{N,T}
$$ 
will be a bounded linear operator. And following Navascu\'es's terminology, we call it a fractal operator on $\text{dom}(\mathcal{E})$.

Now, let us explore some important properties of this operator.

\begin{theorem}\label{thm33}
  Let $Id$ be the identity operator on $\text{dom}(\mathcal{E})$. Denote $$\|\alpha\|_{\infty}=\max \left\{\left\|\alpha_w\right\|_{\infty}: w\in I^N\right\}, \|\alpha\|_{\mathcal{E}}=\max \left\{\left\|\alpha_w\right\|_{\mathcal{E}}: w\in I^N\right\}$$ $$\text {and}~~~~ \mathcal{E}(\alpha)=\max\{\mathcal{E}(\alpha_{w}): w\in I^N \}.$$
If $\|\alpha\|_{\infty}^{2}+2\mathcal{E}(\alpha)<\frac{1}{5^N 4}$, then the following statements hold:
    \begin{enumerate}
        \item\label{item321}  Let $f \in \text{dom}(\mathcal{E}) $ be an arbitrary. Then the perturbation error is of the form:
$$
\left\|f^\alpha_{N,T}-f\right\|_{\mathcal{E} } \leq \max\{A,B\}\|f-T f\|_{\mathcal{E}},
$$
where 
$$A=\frac{\left\|\alpha\right\|_{\infty}}{1-\left\|\alpha\right\|_{\infty}}+\frac{(1-\left\|\alpha\right\|_{\infty})^{-1}2\sqrt{5^{N}}\sqrt{2{\mathcal{E}(\alpha)}}}{(\sqrt{1-5^{N} 4\left\|\alpha\right\|_{\infty}^2}-2\sqrt{5^{N}}\sqrt{2{\mathcal{E}(\alpha)}})}$$
and $$B=\frac{\sqrt{5^{N}}4\left\|\alpha\right\|_{\infty}}{\sqrt{1-5^{N} 4\left\|\alpha\right\|_{\infty}^2}-2\sqrt{5^{N}}\sqrt{2{\mathcal{E}(\alpha)}}}.$$
Note that, for $\|\alpha\|_{\mathcal{E}}=0$ or $T=I$, we have $\mathcal{F}^\alpha_{N,T}= Id.$
\item\label{item322} Under the uniform norm on $\text{dom}(\mathcal{E})$, the fractal operator $\mathcal{F}^\alpha_{N,T}$ is a bounded linear operator. Moreover, the operator norm holds
$$
\left\|\mathcal{F}^\alpha_{N,T}\right\| \leq 1+\max\{A,B\}\left\|Id-T \right\|.
$$
\item\label{item323} If $\left(\left\|\alpha\right\|_{\infty}+2\sqrt{2\cdot 5^{N}} \|\alpha\|_{\mathcal{E}}\right)<\|T\|^{-1}$, then $\mathcal{F}^\alpha_{N,T}$ is bounded below. In particular, $\mathcal{F}^\alpha_{N,T}$ is injective.
\item\label{item324} If $\left(\left\|\alpha\right\|_{\infty}+2\sqrt{2\cdot 5^{N}} \|\alpha\|_{\mathcal{E}}\right)<\min\{1,\|T\|^{-1}\}$ then the inverse of $\mathcal{F}^\alpha_{N,T}$ exists which is also bounded and therefore a topological automorphism. Furthermore,
$$
\frac{1-\left(\left\|\alpha\right\|_{\infty}+2\sqrt{2\cdot 5^{N}} \|\alpha\|_{\mathcal{E}}\right)}{1+\left(\left\|\alpha\right\|_{\infty}+2\sqrt{2\cdot 5^{N}} \|\alpha\|_{\mathcal{E}}\right)\|T\|}\leq\left\|\left(\mathcal{F}^\alpha_{N,T}\right)^{-1}\right\| \leq \frac{1+\left(\left\|\alpha\right\|_{\infty}+2\sqrt{2\cdot 5^{N}} \|\alpha\|_{\mathcal{E}}\right)}{1-\left(\left\|\alpha\right\|_{\infty}+2\sqrt{2\cdot 5^{N}} \|\alpha\|_{\mathcal{E}}\right)\|T\|}.
$$
\item\label{item325} If $\|\alpha\|_{\infty} \neq 0$, then the fixed points of $T$ are also the the fixed points of $\mathcal{F}^\alpha_{N,T}$. If at least one $\alpha_{\omega}$ is nowhere zero, then the fixed points of $\mathcal{F}^\alpha_{N,T}$ are also the fixed points of $T$.
\item\label{item326} If the point spectrum of $T$ contains $1$,  then $\left\|\mathcal{F}^\alpha_{N,T}\right\|_{\mathcal{E}}\geq 1$.
\item\label{item327} The fractal operator $\mathcal{F}^\alpha_{N,T}$ is not the compact operator in the case of $$\left(\left\|\alpha\right\|_{\infty}+2\sqrt{2\cdot 5^{N}} \|\alpha\|_{\mathcal{E}}\right)<\min\{1,\|T\|^{-1}\}.$$
\item\label{item328} If $\left(\left\|\alpha\right\|_{\infty}+2\sqrt{2\cdot 5^{N}} \|\alpha\|_{\mathcal{E}}\right)<\min\{1,\|T\|^{-1}\}$, then $\mathcal{F}^\alpha_{N,T}$ is Fredholm and its index is $0.$
\end{enumerate}
\end{theorem}
\begin{proof}\noindent
    \begin{enumerate}
        \item Using the self-referential equation we have that
        $$
        f^{\alpha}_{N,b}(x)-f(x)=\alpha_w(L_w^{-1}(x))f^{\alpha}_{N,b}(L_w^{-1}(x))- \alpha_w(L_w^{-1}(x))b(L_w^{-1}(x))\; \forall\; x \in L_w(SG), w \in I^N.
        $$
From the definition we have  
$$
\left\|f^\alpha_{N,T}-f\right\|_{\mathcal{E}} \leq \|f^\alpha_{N,T}-f\|_{\infty} + \sqrt{\mathcal{E}(f^\alpha_{N,T}-f)}.
$$
It follows that 
\begin{eqnarray}
  \left|f^\alpha_{N,T}(x)-f(x)\right|
  &=&\left|\alpha_w(L_w^{-1}(x))f^{\alpha}_{N,T}(L_w^{-1}(x))- \alpha_w(L_w^{-1}(x))b(L_w^{-1}(x))\right|\nonumber\\
  &=& \left|\alpha_w(L_w^{-1}(x))\right|\left|f^{\alpha}_{N,T}(L_w^{-1}(x))-b(L_w^{-1}(x))\right|\nonumber\\
  &\leq& \left\|\alpha\right\|_{\infty}\left\|f^{\alpha}_{N,T}-b\right\|_{\infty},\label{eq00}
\end{eqnarray}
which gives that 
\begin{eqnarray*}
 \left\|f^\alpha_{N,T}-f\right\|_{\infty}&\leq&\left\|\alpha\right\|_{\infty}\left\|f^{\alpha}_{N,T}-b\right\|_{\infty}\\&\leq&\left\|\alpha\right\|_{\infty}\left\|f^{\alpha}_{N,T}-f\right\|_{\infty}+\left\|\alpha\right\|_{\infty}\left\|f-b\right\|_{\infty}.
\end{eqnarray*}
So, we have  
\begin{eqnarray}\label{eq11}
    \left\|f^\alpha_{N,T}-f\right\|_{\infty}\leq\frac{\left\|\alpha\right\|_{\infty}}{1-\left\|\alpha\right\|_{\infty}}\left\|f-T f\right\|_{\infty}.
\end{eqnarray}
For $m>N$, we have  
\begin{eqnarray*}
    &&{\mathcal{E}_{m}(f^\alpha_{N,T}-f)}\\&=&  \Big(\frac{5}{3} \Big)^m \sum_{x \sim_m y}\Big((f^\alpha_{N,T}-f)(x)-(f^\alpha_{N,T}-f)(y)\Big)^2\\
    &=&   \Big(\frac{5}{3} \Big)^m \sum_{x \sim_m y}\Big(\alpha_w(L_w^{-1}(x))(f^{\alpha}_{N,T}-b)(L_w^{-1}(x))- \alpha_w(L_w^{-1}(y))(f^{\alpha}_{N,T}-b)(L_w^{-1}(y))\Big)^2\\  \\&\leq&2\left\|\alpha\right\|_{\infty}^2  \Big(\frac{5}{3} \Big)^m \sum_{x \sim_m y}\Big((f^{\alpha}_{N,T}-b)(L_w^{-1}(x))-(f^{\alpha}_{N,T}-b)(L_w^{-1}(y))\Big)^2\\
&+& 2\left\|f^{\alpha}_{N,T}-b\right\|_{\infty}^2   \Big(\frac{5}{3} \Big)^m \sum_{x \sim_m y}\Big(\alpha_w(L_w^{-1}(x))-\alpha_w (L_w^{-1}(y))\Big)^2\\
    &\leq&2\left\|\alpha\right\|_{\infty}^2   \Big(\frac{5}{3} \Big)^m \sum_{x \sim_m y}\Big((f^{\alpha}_{N,T}- f+f-b)(L_w^{-1}(x))-(f^{\alpha}_{N,T}-f+f- b)(L_w^{-1}(y))\Big)^2\\
&+& 2\left\|f^{\alpha}_{N,T}-b\right\|_{\infty}^2   \Big(\frac{5}{3} \Big)^m \sum_{x \sim_m y}\Big(\alpha_w(L_w^{-1}(x))-\alpha_w (L_w^{-1}(y))\Big)^2\\
 &\leq&2\left\|\alpha\right\|_{\infty}^2   \Big(\frac{5}{3} \Big)^m \sum_{x \sim_m y}\Big(\big((f^{\alpha}_{N,T}- f)(L_w^{-1}(x))-(f^{\alpha}_{N,T}-f)(L_w^{-1}(y))\big)\\&&+\big((f- b)(L_w^{-1}(x))-(f-b)(L_w^{-1}(y))\big)\Big)^2\\&&+ 2\left\|f^{\alpha}_{N,T}-b\right\|_{\infty}^2 \Big(\frac{5}{3} \Big)^m \sum_{x \sim_m y}\Big(\alpha_w(L_w^{-1}(x))-\alpha_w (L_w^{-1}(y))\Big)^2\\
    &\leq& 5^{N} 4 \left\|\alpha\right\|_{\infty}^2 \left({\mathcal{E}_{m-N}(f^\alpha_{N,T}-f)}+ {\mathcal{E}_{m-N}(f-b)}\right)+5^{N} 2\left\|f^{\alpha}_{N,T}-b\right\|_{\infty}^2 {\mathcal{E}_{m-N}(\alpha_w)}.
\end{eqnarray*}
Thus, we can conclude that
\begin{eqnarray*}
    {\mathcal{E}(f^\alpha_{N,T}-f)}&\leq& \frac{5^{N} 4\left\|\alpha\right\|_{\infty}^2}{1-5^{N} 4\left\|\alpha\right\|_{\infty}^2}{\mathcal{E}(f-b)}+ \frac{5^{N} 2\left\|f^{\alpha}_{N,T}-b\right\|_{\infty}^2}{1-5^{N} 4\left\|\alpha\right\|_{\infty}^2} {\mathcal{E}(\alpha_w)}
\end{eqnarray*}
and
    \begin{eqnarray*}
    \sqrt{\mathcal{E}(f^\alpha_{N,T}-f)}&\leq& \frac{\sqrt{5^{N}}4\left\|\alpha\right\|_{\infty}}{\sqrt{1-5^{N}4\left\|\alpha\right\|_{\infty}^2}}\sqrt{\mathcal{E}(f-b)}+\frac{2\sqrt{5^{N}}\sqrt{2}\left\|f^{\alpha}_{N,T}-b\right\|_{\infty}}{\sqrt{1-5^{N} 4\left\|\alpha\right\|_{\infty}^2}} \sqrt{{\mathcal{E}(\alpha_w)}}.\\
\end{eqnarray*}
Which implies that 
\begin{eqnarray*}
    &&\left\|f^\alpha_{N,T}-f\right\|_{\mathcal{E}} \\& = & \|f^\alpha_{N,T}-f\|_{\infty} + \sqrt{\mathcal{E}(f^\alpha_{N,T}-f)}\\
    &\leq& \frac{\left\|\alpha\right\|_{\infty}}{1-\left\|\alpha\right\|_{\infty}}\left\|f-T f\right\|_{\infty}+\frac{\sqrt{5^{N}}4\left\|\alpha\right\|_{\infty}}{\sqrt{1-5^{N}4\left\|\alpha\right\|_{\infty}^2}}\sqrt{\mathcal{E}(f-T f)}+\frac{2\sqrt{5^{N}}\sqrt{2}\left\|f^{\alpha}_{N,T}-T f\right\|_{\infty}}{\sqrt{1-5^{N} 4\left\|\alpha\right\|_{\infty}^2}} \sqrt{{\mathcal{E}(\alpha_w)}}\\&\leq & 
   \left(\frac{\left\|\alpha\right\|_{\infty}}{1-\left\|\alpha\right\|_{\infty}}+\frac{2\sqrt{5^{N}}\sqrt{2 {\mathcal{E}(\alpha)}}}{\sqrt{1-5^{N}4\left\|\alpha\right\|_{\infty}^2}}\right)\left\|f-T f\right\|_{\infty}+\frac{\sqrt{5^{N}}4\left\|\alpha\right\|_{\infty}}{\sqrt{1-5^{N}4\left\|\alpha\right\|_{\infty}^2}}\sqrt{\mathcal{E}(f-T f)}\\&& +\frac{2\sqrt{5^{N}}\sqrt{2{\mathcal{E}(\alpha)}}}{\sqrt{1-5^{N} 4\left\|\alpha\right\|_{\infty}^2}} {\left\|f^{\alpha}_{N,T}- f\right\|_{\mathcal{E}}} 
\\&\leq& \left(1-\frac{2\sqrt{5^{N}}\sqrt{2{\mathcal{E}(\alpha)}}}{\sqrt{1-5^{N} 4\left\|\alpha\right\|_{\infty}^2}}\right)^{-1}\max\left\{\left(\frac{\left\|\alpha\right\|_{\infty}}{1-\left\|\alpha\right\|_{\infty}}+\frac{2\sqrt{5^{N}}\sqrt{2{\mathcal{E}(\alpha)}}}{\sqrt{1-5^{N}4\left\|\alpha\right\|_{\infty}^2}}\right), \frac{\sqrt{5^{N}}4\left\|\alpha\right\|_{\infty}}{\sqrt{1-5^{N}4\left\|\alpha\right\|_{\infty}^2}}\right\}\\&&\times \Big(\left\|f-T f\right\|_{\infty}+\sqrt{\mathcal{E}(f-T f)}\Big)\\
&=& \max\{A, B\}\|f-T f\|_{\mathcal{E}},
\end{eqnarray*}
where $$A=\frac{\left\|\alpha\right\|_{\infty}}{1-\left\|\alpha\right\|_{\infty}}+\frac{(1-\left\|\alpha\right\|_{\infty})^{-1}2\sqrt{5^{N}}\sqrt{2{\mathcal{E}(\alpha)}}}{\left(\sqrt{1-5^{N} 4\left\|\alpha\right\|_{\infty}^2}-2\sqrt{5^{N}}\sqrt{2{\mathcal{E}(\alpha)}}\right)}$$ and $$ B=\frac{\sqrt{5^{N}}4\left\|\alpha\right\|_{\infty}}{\sqrt{1-5^{N} 4\left\|\alpha\right\|_{\infty}^2}-2\sqrt{5^{N}}\sqrt{2{\mathcal{E}(\alpha)}}}.$$
\item By the first assertion, we know that
$$
\left\|f^\alpha_{N,T}\right\|_{\mathcal{E}}-\left\|f\right\|_{\mathcal{E}}\leq\left\|f^\alpha_{N,T}-f\right\|_{\mathcal{E}}\leq\max\{A,B\}\left\|f-T f\right\|_{\mathcal{E}}
$$
and
$$
\left\|\mathcal{F}^{\alpha}_{N,T}(f)\right\|_{\mathcal{E}}-\left\|f\right\|_{\mathcal{E}}\leq\left\|f^\alpha_{N,T}-f\right\|_{\mathcal{E}}\leq\max\{A,B\}\left\|f-T f\right\|_{\mathcal{E}}.
$$
It follows from this that
$$
\left\|\mathcal{F}^{\alpha}_{N,T}(f)\right\|_{\mathcal{E}}\leq\left[1+\max\{A,B\}\left\|Id-T \right\|\right]\left\|f \right\|_{\mathcal{E}},
$$
which implies that 
$$
\left\|\mathcal{F}^{\alpha}_{N,T}\right\|\leq1+\max\{A,B\}\left\|Id-T \right\|.
$$
\item Assume that $m>N$. Then, have 
\begin{eqnarray*}
     &&{\mathcal{E}_{m}(f^\alpha_{N,T}-f)}\\&=&  \Big(\frac{5}{3} \Big)^m \sum_{x \sim_m y}\Big((f^\alpha_{N,T}-f)(x)-(f^\alpha_{N,T}-f)(y)\Big)^2\\
    &=&   \Big(\frac{5}{3} \Big)^m \sum_{x \sim_m y}\Big(\alpha_w(L_w^{-1}(x))(f^{\alpha}_{N,T}-b)(L_w^{-1}(x))- \alpha_w(L_w^{-1}(y))(f^{\alpha}_{N,T}-b)(L_w^{-1}(y))\Big)^2\\  \\&\leq&2\left\|\alpha\right\|_{\infty}^2  \Big(\frac{5}{3} \Big)^m \sum_{x \sim_m y}\Big((f^{\alpha}_{N,T}-b)(L_w^{-1}(x))-(f^{\alpha}_{N,T}-b)(L_w^{-1}(y))\Big)^2\\
&+& 2\left\|f^{\alpha}_{N,T}-b\right\|_{\infty}^2   \Big(\frac{5}{3} \Big)^m \sum_{x \sim_m y}\Big(\alpha_w(L_w^{-1}(x))-\alpha_w (L_w^{-1}(y))\Big)^2\\&\leq & 2 \left\|\alpha\right\|_{\infty}^2 5^{N} \mathcal{E}_{m-N}(f^{\alpha}_{N,T}-b)+2 \left\|f^{\alpha}_{N,T}-b\right\|_{\infty}^2 5^{N} \mathcal{E}_{m-N}(\alpha_w)
.\end{eqnarray*}
This implies that 
$$\mathcal{E}(f^\alpha_{N,T}-f)\leq 2 \left\|\alpha\right\|_{\infty}^2 5^{N} \mathcal{E}(f^{\alpha}_{N,T}-b)+2 \left\|f^{\alpha}_{N,T}-b\right\|_{\infty}^2 5^{N} \mathcal{E}(\alpha_w).$$
Thus, we get
$$\sqrt{\mathcal{E}(f^\alpha_{N,T}-f)}\leq 2\sqrt{2\cdot 5^{N}} \left\|\alpha\right\|_{\infty}  \sqrt{\mathcal{E}(f^{\alpha}_{N,T}-T f)}+2\sqrt{2\cdot 5^{N} } \left\|f^{\alpha}_{N,T}-T f\right\|_{\infty} \sqrt{\mathcal{E}(\alpha)}.$$
It follows from \eqref{eq00} that
$$
\left\|f^\alpha_{N,T}-f\right\|_{\infty}\leq\left\|\alpha\right\|_{\infty}\left\|f^{\alpha}_{N,T}-T f\right\|_{\infty}.
$$
This implies that
\begin{eqnarray*}
    &&\left\|f\right\|_{\mathcal{E}}-\left\|f^\alpha_{N,T}\right\|_{\mathcal{E}}\\&\leq& \left\|f^\alpha_{N,T}-f\right\|_{\mathcal{E}}\\& = & \left\|f^\alpha_{N,T}-f\right\|_{\infty}+  \sqrt{\mathcal{E}(f^\alpha_{N,T}-f)}\\&\leq&
\left\|\alpha\right\|_{\infty}\left\|f^{\alpha}_{N,T}-T f\right\|_{\infty}+2\sqrt{2\cdot 5^{N}} \left\|\alpha\right\|_{\infty}  \sqrt{\mathcal{E}(f^{\alpha}_{N,T}-T f)}+2\sqrt{2\cdot 5^{N} } \left\|f^{\alpha}_{N,T}-T f\right\|_{\infty} \sqrt{\mathcal{E}(\alpha)}\\&=&
\bigg(\left\|\alpha\right\|_{\infty}+2\sqrt{2\cdot 5^{N} }\sqrt{\mathcal{E}(\alpha)}\bigg)\left\|f^{\alpha}_{N,T}-T f\right\|_{\infty}+2\sqrt{2\cdot 5^{N}} \left\|\alpha\right\|_{\infty}  \sqrt{\mathcal{E}(f^{\alpha}_{N,T}-T f)}
    \\&=& (\left\|\alpha\right\|_{\infty}+2\sqrt{2\cdot 5^{N}} \|\alpha\|_{\mathcal{E}})\left\|f^{\alpha}_{N,T}-T f\right\|_{\mathcal{E}}\\&\leq& 
(\left\|\alpha\right\|_{\infty}+2\sqrt{2\cdot 5^{N}} \|\alpha\|_{\mathcal{E}})\left(\left\|f^\alpha_{N,T}\right\|_{\mathcal{E}}+\left\|T\right\|\left\|f\right\|_{\mathcal{E}}\right)
\end{eqnarray*}
and then 
$$
\Big(1-(\left\|\alpha\right\|_{\infty}+2\sqrt{2\cdot 5^{N}} \|\alpha\|_{\mathcal{E}})\left\|T\right\|\Big)\left\|f\right\|_{\mathcal{E}}\leq \Big(1+(\left\|\alpha\right\|_{\infty}+2\sqrt{2\cdot 5^{N}} \|\alpha\|_{\mathcal{E}})\Big)\left\|f^\alpha_{N,T}\right\|_{\mathcal{E}}.
$$
It follows from $(\left\|\alpha\right\|_{\infty}+2\sqrt{2\cdot 5^{N}} \|\alpha\|_{\mathcal{E}})\|T\|<{1}$ that
$$
\left\|f\right\|_{\mathcal{E}}\leq \frac{1+(\left\|\alpha\right\|_{\infty}+2\sqrt{2\cdot 5^{N}} \|\alpha\|_{\mathcal{E}})}{1-(\left\|\alpha\right\|_{\infty}+2\sqrt{2\cdot 5^{N}} \|\alpha\|_{\mathcal{E}})\left\|T\right\|}\left\|f^\alpha_{N,T}\right\|_{\mathcal{E}}.
$$
The proof of assertions (\ref{item324}),(\ref{item325}),(\ref{item326}),(\ref{item327}),(\ref{item328}) will follow in similar way as in Theorem \ref{thm32} and is therefore
omitted.
    \end{enumerate}
\end{proof}

\begin{remark}
    Here we compare our result with \cite[Corollary 3.6]{SP}. If  $\alpha$ is a constant function, then in the proof of \cite{Mverma1}, using $(r+t)^2 \le 2 r^2 +2 t^2$, we get with the notation of \cite{Mverma1}:
    \begin{equation*}
   \begin{aligned} 
   &|\mathcal{T}f(x)-\mathcal{T}f(y)|^2  \le  ~2 \|\alpha\|_{\infty}^2 |f(L_w^{-1}(x))-f(L_w^{-1}(y))|^2+ 2|q_w(L_w^{-1}(x))-q_w(L_w^{-1}(y))|^2.
   \end{aligned}
   \end{equation*}
    So, finally, we obtain
$\| \alpha\|_{\infty} < \frac{1}{2\sqrt{5^N}}$ for the above theorem to hold. This improves the bound
obtained in the setting of the real-valued $\alpha$-fractal function on SG in \cite[Corollary 3.6]{SP}.
\end{remark}
\subsection{Graphs of some $\alpha$-fractal function}    
\begin{figure}[h!]
\begin{minipage}{0.5\textwidth}
\includegraphics[width=1.0\linewidth]{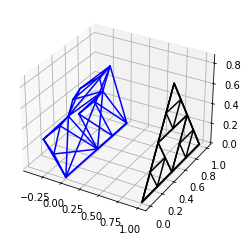}
{\hspace*{1cm}Graph of $f^{\alpha}$ on $2$nd level}
\end{minipage}\hspace*{0.3cm}
\begin{minipage}{0.5\textwidth}\includegraphics[width=1.0\linewidth]{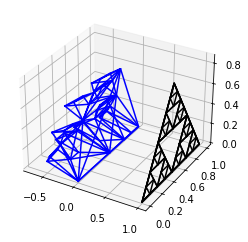}
{\hspace*{2cm}Graph of $f^{\alpha}$ on $3$rd level}
\end{minipage}
\begin{minipage}{0.5\textwidth}
\includegraphics[width=1.0\linewidth]{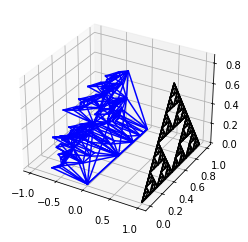}
{\hspace*{1cm}Graph of $f^{\alpha}$ on $4$th level}
\end{minipage}\hspace*{0.3cm}
\begin{minipage}{0.5\textwidth}
\includegraphics[width=1.0\linewidth]{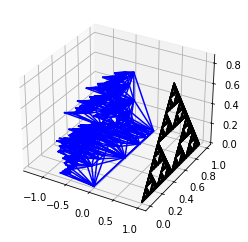}
{\hspace*{2cm}Graph of $f^{\alpha}$ on $5$th level}
\end{minipage}
\caption{$f(x,y)=\frac{y^2\sin{(x)}}{2}$, $b(x,y)=x(x-1)(y-\frac{\sqrt{3}}{2})\frac{y^2\sin{(0.5)}}{2}$ and   $\alpha=0.9$.}
\end{figure} 
\subsection{Lebesgue spaces on SG}
Next, we now define a function space called the $\mathcal{L}_q$ space on $SG$, which contains all the functions that are $q$-integrable functions on $SG$. 

First, recall that $V_0=\{p_1,p_2,p_3\}$ is the set of vertices of an equilateral triangle in the plane $\mathbb{R}^2$. For each $i=1,2,3,$ define the map 
$$
L_i(x)=\frac{1}{2}( x + p_i).
$$
The collection of these three maps forms an IFS $\mathcal{J}=\{\mathbb{R}^2; L_i, i=1,2,3\}$. The $SG$ is the unique compact set (called the attractor of the IFS) that satisfies:
$$
SG=\bigcup_{i=1}^{3}L_i(SG).
$$ 
And there exists a unique Borel probability measure $\nu_p$ supported on the $SG$ such that
$$
\nu_p=\sum_{i=1}^{3} p_i \nu_p\circ L_i^{-1},
$$
where $p=(p_1,p_2,p_3)$ be a probability vector with $\sum_{i=1}^{3} p_i =1$ and $p_i>0$.

Now, we define the space $\mathcal{L}_q(SG,\nu_p)$, which consists of all real-valued functions $f$ on $SG$ such that
$$
\int_{SG}|f(x)|^qd\nu_p(x) < \infty.
$$
The space $\mathcal{L}_q(SG,\nu_p)$ becomes the Banach space with the norm 
$$
\|f\|_{q}=\bigg(\int_{SG}|f(x)|^qd\nu_p(x)\bigg)^{\frac{1}{q}}.
$$
\begin{note}
For $p\ne p'$ (which yields $\nu_p \neq \nu_{p'}$), we can get 
\[
\text{neither}~ \mathcal{L}_q(SG,\nu_p) \subseteq \mathcal{L}_q(SG,\nu_{p'}) \text{ nor } \mathcal{L}_q(SG,\nu_{p'}) \subseteq \mathcal{L}_q(SG,\nu_{p}).
\]
For example:  Choose $p=(\tfrac{1}{3},\tfrac{1}{3},\tfrac{1}{3})$ and $p'=(\tfrac{1}{8},\tfrac{3}{4},\tfrac{1}{8})$ two probability vectors. Recall that the IFS $\{SG; L_1, L_2, L_3\}$ consisting of similarity mappings satisfies the open set condition. Then, using \cite[Theorem $2.2$]{Moran}, invariant measure $\nu_{p}$ corresponding to $p$ is $\nu_{p}=\mathcal{H}^{\frac{\log 3}{\log 2}}\vert_{SG}$. We know that $\dim_{H}(\nu_{p})=\dim_{H}(SG)=\displaystyle\frac{\log 3}{\log 2}$ and \[\dim_{H}(\nu_{p'})=\displaystyle\frac{ - \tfrac{1}{8} \log (8)+\tfrac{3}{4}\log (3/4)- \tfrac{1}{8} \log (8)}{- \tfrac{1}{8} \log (2)-\tfrac{3}{4}\log (2)- \tfrac{1}{8} \log (2)}= \displaystyle\frac{ \tfrac{3}{4}\log (\frac{3}{4}) - \frac{3}{4} \log (2)}{ - \log (2)}  \approx 1.0612781245 < \frac{\log 3}{\log 2}.\]
Therefore, for $\dim_{H}(\nu_{p'})<\gamma< \frac{\log 3}{\log 2}$ there exists a set $D_{\gamma} \subset SG$ such that $\dim_{H}D_{\gamma}=\gamma$ and $\nu_{p'}(D_{\gamma})>0$ but $\nu_{p}(D_{\gamma})=0$. Now define a map
\[
f^*(x)=\begin{cases}
\infty, & \text{if } x \in D_{\gamma}\\ \\
0, & \text{otherwise}.
\end{cases}
\]
  It is simple to prove that $f^*\in \mathcal{L}_q(SG,\nu_{p})$ but not in $\mathcal{L}_q(SG,\nu_{p'})$.


Now for other part, from \cite[Theorem $2.1$]{Moran}, $\nu_{p}$ and $\nu_{p'} $ are mutually singular measures for $p\ne p'$, that is, there exists a set $A \subset SG$ such that $\nu_{p} (A) =1, \nu_{p'} (A) =0$. Consider a partition $A= \cup_{n=1}^{\infty} A_n$, where $A_n$ are disjoint subsets of $A$ with $\nu_{p} (A_n) >0$.
Define, 
\[g^*(x)=\begin{cases}
   n^{\frac{1}{q}}, & \text{if } x \in A_n \\
   0, & \text{if } x \in SG \setminus A.
\end{cases}\]
Then, $\int_{SG} |g^*(x)|^q d\nu_p(x) = \sum_{n=1}^{\infty} n. \nu_{p}(A_n)= \infty$ and $\int_{SG} |g^*(x)|^q d\nu_{p'}(x) = \sum_{n=1}^{\infty} n. \nu_{p'}(A_n)= 0$.
Therefore, $g^*\notin \mathcal{L}_q(SG,\nu_p)$ but $g^*\in\mathcal{L}_q(SG,\nu_{p'})$. 
This completes our task.
\end{note}

\par

For a given $f \in \mathcal{L}_q(SG,\nu_p)$, if we consider $q_{w}: SG \to \mathbb{R}$ as 
$$q_w(x)=f(L_w(x))-\alpha_w(x)b(x),$$ where $b \in \mathcal{L}_q(SG,\nu_p)$ with  $b|_{V_0}=f|_{V_0}$. Following the idea of  \cite[Theorem 9]{Mverma1}, we prove our next result.

     \begin{theorem} Let $N\in \mathbb{N}$ and
     let $ b \in \mathcal{L}_q(SG)$ and $\alpha_w \in \mathcal{L}_{\infty}(SG)$ for $w\in I^n$ such that $\sum_{w \in I^N} p_w \|\alpha_w\|_{\mathcal{L}_{\infty}(SG)}^q <1$. Then there exists a unique $f^{\alpha}_{N,b} \in \mathcal{L}_q(SG)$ such that $f^{\alpha}_{N,b}|_{V_N}=f|_{V_N}$ and  satisfies the following self-referential equation:
$$
f^{\alpha}_{N,b}(x)=f(x)+\alpha_w(L_w^{-1}(x))f^{\alpha}_{N,b}(L_w^{-1}(x))- \alpha_w(L_w^{-1}(x))b(L_w^{-1}(x))\; \forall\; x \in L_w(SG), w \in I^N.
$$    
\end{theorem}
\begin{proof}
Let us define 
$$ \mathcal{L}_q^*(SG,\nu_p )= \{ f \in \mathcal{L}_q(SG,\nu_p ): f|_{V_0}=g_*|_{V_0}\}.
$$
which means the set of all functions in  $\mathcal{L}_q^*(SG,\nu_p )$ that agree with a fixed function $g_*$ on the set $V_0$.  It is easy to check that the set $\mathcal{L}_q^*(SG,\nu_p )$ is closed in $\mathcal{L}_q(SG,\nu_p).$ Since $(\mathcal{L}_q(SG,\nu_p),\|.\|_{q})$ is a Banach space, this implies that $\mathcal{L}_q^*(SG,\nu_p )$ is also complete with respect with respect to same norm $\|.\|_{q}$.

Now, we define the RB operator $\mathcal{T}: \mathcal{L}_q^*(SG,\nu_p ) \rightarrow \mathcal{L}_q^*(SG,\nu_p )$ as follows  
$$
(\mathcal{T}g)(x)= \alpha_w(L_w^{-1}(x))g(L_w^{-1}(x))+f(x)-\alpha_w(L_w^{-1}(x))b(L_w^{-1}(x)))\; \forall\; x \in L_w(SG), w \in I^n
$$
This operator $\mathcal{T}$ is well-defined.
Let $f, g \in \mathcal{L}_q^*(SG,\nu_p )$. Then
               \begin{equation*}
                    \begin{aligned}
                                \big|(\mathcal{T}f)(x) -(\mathcal{T}g)(x)\big| &= \Big|\alpha_w(L_w^{-1}(x))f(L_w^{-1}(x))-\alpha_w(L_w^{-1}(x))g(L_w^{-1}(x))\Big|\\
                                &= | \alpha_w(L_w^{-1}(x)) ~(f-g)(L_w^{-1}(x))|\\
                       &= |\alpha_w(L_w^{-1}(x))| ~|(f-g)(L_w^{-1}(x))|\\
                       &\le \|\alpha_w\|_{\mathcal{L}_{\infty}(SG)} |(f-g)(L_w^{-1}(x))|.
                    \end{aligned}
                    \end{equation*}
            This inequality holds for almost all (with respect to $\nu_p$) $x \in L_w(SG)$ and  $ \forall~~ w \in I^n.$ Now, we have
           \begin{align*}
             \int_{SG} \big|(\mathcal{T}f)(x) -(\mathcal{T}g)(x)\big|^q d\nu_p(x)
                                   & \le  \sum_{w \in I^n} \int_{L_w(SG)}  \|\alpha_w\|_{\mathcal{L}_{\infty}(SG)}^q|(f-g)(L_w^{-1}(x))|^q  d\nu_p(x) . \end{align*}
 By using the self-similarity of  measure $\nu_p=\sum_{\sigma\in I^N}p_\sigma\nu_p\circ L_\sigma^{-1}$ and \cite[Theorem 2.1]{Bandt}, we get          
           \begin{align*}
             \int_{SG} \big|&(\mathcal{T}f)(x) -(\mathcal{T}g)(x)\big|^q d\nu_p(x)
                                  \\ & \le  \sum_{w \in I^N} \int_{L_w(SG)} \|\alpha_w\|_{\mathcal{L}_{\infty}(SG)}^q|(f-g)(L_w^{-1}(x))|^q  d\bigg(\sum_{\sigma\in I^N}p_\sigma \nu_p\circ L_\sigma^{-1}(x)\bigg)             \\& = \sum_{w \in I^N} p_w \|\alpha\|_{\mathcal{L}_{\infty}(SG)}^q  \int_{L_w(SG)}|(f-g)(L_w^{-1}(x))|^q   d(\nu_p\circ L_w^{-1}(x))\\&=\sum_{w \in I^N} p_w \|\alpha_w\|_{\mathcal{L}_{\infty}(SG)}^q \int_{SG}|(f-g)(\tilde{x})|^q   d\nu_p(\tilde{x})
                                  \\
                             &=  \sum_{w \in I^N} p_w \|\alpha_w\|_{\mathcal{L}_{\infty}(SG)}^q \int_{SG}|(f-g)(\tilde{x})|^q   d\nu_p(\tilde{x}).
              \end{align*} 
       Thus, we obtain $$\|\mathcal{T}f-\mathcal{T}g\|_{q} \le \Big(\sum_{w \in I^N} p_w \|\alpha_w\|_{\mathcal{L}_{\infty}(SG)}^q\Big)^{\frac{1}{q}} \|f-g\|_{q}.$$ Since $\sum_{w \in I^N} p_w \|\alpha_w\|_{\mathcal{L}_{\infty}(SG)}^q  < 1$, we conclude that $\mathcal{T}$ is a contraction operator on $ \mathcal{L}_q^*(SG,\nu_p ).$ By the application of Banach contraction mapping principle,  $\mathcal{T}$ has a unique fixed point, namely $f^{\alpha}_{N,b} \in \mathcal{L}_q^*(SG,\nu_p)$. Therefore, we conclude that $f^{\alpha}_{N,b} \in \mathcal{L}_q(SG,\nu_p).$ This completes the proof.
       \end{proof}

Now, for a given bounded linear operator $T:\mathcal{L}_q(SG) \to \mathcal{L}_q(SG) $ such that $(Tg)|{V_0}=g|{V_0},$ if we consider the base function $b=Tf.$ Then a map $\mathcal{F}^{\alpha}_{N,T}: \mathcal{L}_q(SG) \to \mathcal{L}_q(SG)$ defined as $\mathcal{F}^{\alpha}_{N,T}(f)=f^{\alpha}_{N,b}=f^{\alpha}_{N,T}$ will be bounded linear operator. This map $\mathcal{F}^{\alpha}_{N,T}$, following Navascu\'es's terminology, we call a fractal operator on $\mathcal{L}_q(SG).$

Now we collect several properties of this operator.

\begin{theorem}
    Let $Id$ be the identity operator on $\mathcal{L}_q(SG,\nu_p)$. Denote $$\big\|\alpha\big\|_{\mathcal{L}_{\infty}(SG)}=\max \left\{\big\|\alpha_w\big\|_{\mathcal{L}_{\infty}(SG)}: w\in I^N\right\}.$$ If $\big\|\alpha\big\|_{\mathcal{L}_{\infty}(SG)}<1$, then the following statements hold:
    \begin{enumerate}
        \item\label{item361}  Let $f \in \mathcal{L}_q(SG,\nu_p) $ be an arbitrary. Then the perturbation error is of the form:
$$
\big\|f^\alpha_{N,T}-f\big\|_q\leq\frac{\big\|\alpha\big\|_{\mathcal{L}_{\infty}(SG)}}{1-\big\|\alpha\big\|_{\mathcal{L}_{\infty}(SG)}}\big\|f-Tf\big\|_q.
$$

Note that, for $\big\|\alpha\big\|_{\mathcal{L}_{\infty}(SG)}=0$ or $T=I$, we have $\mathcal{F}^\alpha_{N,T}= Id.$
\item\label{item362} Under the uniform norm on $\mathcal{L}_q(SG,\nu_p)$, the fractal operator $\mathcal{F}^\alpha_{N,T}$ is a bounded linear operator. Moreover, the operator norm holds
$$\big\|\mathcal{F}^\alpha_{N,T}\big\|\leq1+\frac{\big\|\alpha\big\|_{\mathcal{L}_{\infty}(SG)}\big\|I-T\big\|}{1-\big\|\alpha\big\|_{\mathcal{L}_{\infty}(SG)}}.$$
\item\label{item363} If $\|\alpha\|_{\mathcal{L}_{\infty}(SG)}<\|T\|^{-1}$ then $\mathcal{F}^\alpha_{N,T}$ is bounded below. In particular, $\mathcal{F}^\alpha_{N,T}$ is injective.
\item\label{item364} If $\|\alpha\|_{\mathcal{L}_{\infty}(SG)}<\|T\|^{-1}$ then the inverse of $\mathcal{F}^\alpha_{N,T}$ exists which is also bounded and therefore a topological automorphism. Furthermore,
$$
\frac{1-\|\alpha\|_{\mathcal{L}_{\infty}(SG)}}{1+\|\alpha\|_{\mathcal{L}_{\infty}(SG)}\|T\|}\leq\left\|\left(\mathcal{F}^\alpha_{N,T}\right)^{-1}\right\| \leq \frac{1+\|\alpha\|_{\mathcal{L}_{\infty}(SG)}}{1-\|\alpha\|_{\mathcal{L}_{\infty}(SG)}\|T\|}.
$$
\item\label{item365} If $\|\alpha\|_{\mathcal{L}_{\infty}(SG)}\neq 0$ then the fixed points of $T$ are also the the fixed points of $\mathcal{F}^\alpha_{N,T}$. If at least one $\alpha_{\omega}$ is nowhere zero, then the fixed points of $\mathcal{F}^\alpha_{N,T}$ are also the fixed points of $T$.
\item\label{item366} If the point spectrum of $T$ contains $1$ , then $\left\|\mathcal{F}^\alpha_{N,T}\right\|\geq 1$.
\item\label{item367} The fractal operator $\mathcal{F}^\alpha_{N,T}$ is not the compact operator in the case of $\|\alpha\|_{\mathcal{L}_{\infty}(SG)}<\|T\|^{-1}$.
\item\label{item368} If $\|\alpha\|_{\mathcal{L}_{\infty}(SG)}<\|T\|^{-1}$, then $\mathcal{F}^\alpha_{N,T}$ is Fredholm and its index is $0.$
\end{enumerate}
\end{theorem}
\begin{proof}
\begin{enumerate}
    \item It follows from the self-referential equation that
    $$f^\alpha_{N,T}(x)=f(x)
          +\alpha_w(L_w^{-1}(x))f^{\alpha}_{N,T}(L_w^{-1}(x))- \alpha_w(L_w^{-1}(x))b(L_w^{-1}(x))$$
    We have
    \begin{eqnarray*}
        \left|f^\alpha_{N,T}(x)-f(x)\right|^q
          &=&\left|\alpha_w(L_w^{-1}(x))f^{\alpha}_{N,T}(L_w^{-1}(x))- \alpha_w(L_w^{-1}(x))b(L_w^{-1}(x))\right|^q\\
          &\leq& \|\alpha_w\|_{\mathcal{L}_{\infty}(SG)}^q\left|(f^{\alpha}_{N,T}-b)(L_w^{-1}(x))\right|^q\\
    \end{eqnarray*}
    and
    \begin{align*}
        \int_{SG} \big|&f^\alpha_{N,T}(x)-f(x)\big|^q d\nu_p(x)
           \\ & \le \|\alpha\|_{\mathcal{L}_{\infty}(SG)}^q \sum_{w \in I^N} \int_{L_w(SG)}|(f^{\alpha}_{N,T}-b)(L_w^{-1}(x))|^q  d\bigg(\sum_{\sigma\in I^n}p_\sigma \nu_p\circ L_\sigma^{-1}(x)\bigg) 
           \\& = \|\alpha\|_{\mathcal{L}_{\infty}(SG)}^q  \sum_{w \in I^N} p_w \int_{L_w(SG)}|(f^{\alpha}_{N,T}-b)(L_w^{-1}(x))|^q   d(\nu_p\circ L_w^{-1}(x))\\&=\|\alpha\|_{\mathcal{L}_{\infty}(SG)}^q \sum_{w \in I^N} p_w \int_{SG}|(f^{\alpha}_{N,T}-b)(\tilde{x})|^q   d\nu_p(\tilde{x})
                  \\
            &=  \|\alpha\|_{\mathcal{L}_{\infty}(SG)}^q \int_{SG}|(f^{\alpha}_{N,T}-b)(\tilde{x})|^q   d\nu_p(\tilde{x}).
    \end{align*}
    So
    \begin{eqnarray*}
        \big\|f^\alpha_{N,T}-f\big\|_q&\leq& \big\|\alpha\big\|_{\mathcal{L}_{\infty}(SG)} \big\|(f^{\alpha}_{N,T}-b)\big\|_q \\
        &\leq& \big\|\alpha\big\|_{\mathcal{L}_{\infty}(SG)} \big\|(f^{\alpha}_{N,T}-f)+(f-b)\big\|_q\\ 
         &\leq& \big\|\alpha\big\|_{\mathcal{L}_{\infty}(SG)} \big\|f^{\alpha}_{N,T}-f\big\|_q+\big\|\alpha\big\|_{\mathcal{L}_{\infty}(SG)}\big\|f-b\big\|_q, 
    \end{eqnarray*}
    by Minkowski inequality.
    Hence 
    $$\big\|f^\alpha_{N,T}-f\big\|_q\leq\frac{\big\|\alpha\big\|_{\mathcal{L}_{\infty}(SG)}}{1-\big\|\alpha\big\|_{\mathcal{L}_{\infty}(SG)}}\big\|f-Tf\big\|_q.$$
    \item By the first assertion, we know
    $$\big\|f^\alpha_{N,T}\big\|_q-\big\|f\big\|_q\leq\big\|f^\alpha_{N,T}-f\big\|_q\leq\frac{\big\|\alpha\big\|_{\mathcal{L}_{\infty}(SG)}}{1-\big\|\alpha\big\|_{\mathcal{L}_{\infty}(SG)}}\big\|f-Tf\big\|_q.$$
    Therefore
    $$\big\|\mathcal{F}^\alpha_{N,T}(f)\big\|_q\leq\Bigg(1+\frac{\big\|\alpha\big\|_{\mathcal{L}_{\infty}(SG)}\big\|I-T\big\|}{1-\big\|\alpha\big\|_{\mathcal{L}_{\infty}(SG)}}\Bigg)\big\|f\big\|_q.$$
    Hence
    $$\big\|\mathcal{F}^\alpha_{N,T}\big\|_q\leq1+\frac{\big\|\alpha\big\|_{\mathcal{L}_{\infty}(SG)}\big\|I-T\big\|}{1-\big\|\alpha\big\|_{\mathcal{L}_{\infty}(SG)}}.$$
   The arguments supporting claims (\ref{item363}), (\ref{item364}), (\ref{item365}), (\ref{item366}), (\ref{item367}), and (\ref{item368}) can be derived using a similar approach as demonstrated in Theorem \ref{thm33}.
\end{enumerate}
\end{proof}

\section{Concluding remarks and possible future aspects}
We have studied several important properties of the fractal operator defined on function spaces. In the future, we plan to explore and describe isometries and surjective isometries on these spaces. We also aim to approximate quantities such as the spectral radius, numerical radius, and numerical range of these fractal operators. In addition to studying isometries of fractal operators, we also plan to explore other properties of the oscillation space, energy space, and Lebesgue spaces, such as whether they are separable, algebraically reflexive, and more, especially in the context of Banach spaces.

\bibliographystyle{amsplain}

\begin{thebibliography}{10}

\bibitem{RS1} R. Achour, B. Selmi, Exploring simultaneous characterizations of the generalized local fractal dimension functions, Discrete and Continuous Dynamical Systems-S (2025) Doi: 10.3934/dcdss.2025037.
\bibitem{RHS1} R. Achour, J. Hattab, B. Selmi, New fractal dimensions of measures and decompositions of singularly continuous measures, Fuzzy Sets and Systems 479 (2024) 108859.

\bibitem{ABP} Amit, V. Basotia, A. Prajapati, Non-stationary $\phi$-contractions and associated fractals, The Journal of Analysis 31(2) (2023) 1375-1391.

\bibitem{ES1} E. Agrawal, S. Verma, Fractal surfaces in H\"older and Sobolev spaces,  J. Anal. 32(2) (2024) 1161-1179.
\bibitem{ES} E. Agrawal, S. Verma, Dimensions and Stability of invariant measures supported on fractal Surfaces, Discrete and Continuous Dynamical Systems - S, (2024) 1-27, doi: 10.3934/dcdss.2024159.
 \bibitem{AV1} E. Agrawal, S. Verma, Dimension preserving approximation and estimation: Fractal surfaces and Riemann-Liouville fractional integrals, Contemporary Mathematics (AMS), Vol. 825, (2025) 1-24.
\bibitem{AV2} E. Agrawal, S. Verma, Dimension preserving set-valued approximation and decomposition via metric sum, (2025) arXiv preprint  1-32, 
doi.org/10.48550/arXiv.2504.11356.

 \bibitem{ASV1} V. Agrawal, T. Som, S. Verma, A note on stability and fractal dimension of bivariate $\alpha$-fractal functions, Numerical Algorithms 93 (4), (2023) 1811-1833.
 \bibitem{AJ1} N. Attia, H. Jebali, On the construction of recurrent fractal interpolation functions using Geraghty contractions. Electronic Research Archive 31(11) (2023) 6866-6880. 
\bibitem{AJ2} N. Attia, H. Jebali, On the box dimension of recurrent fractal interpolation functions defined with Matkowski contractions, J. Anal (2024) doi: 10.1007/s41478-024-00816-2


\bibitem{Bandt}  C. Bandt, M. Barnsley, M. Hegland, A. Vince, Old wine in fractal bottles I: Orthogonal expansions on self-referential spaces via fractal transformations, Chaos, Solitons \& Fractals 91 (2016) 478-489.
\bibitem{MF1} M. F. Barnsley, Fractal Everywhere, Academic Press, San Diego, 1988.
\bibitem{MF2} M. F. Barnsley, Fractal functions and interpolation, Constr. Approx. 2 (1986) 303-329.
\bibitem{MF3} M. F. Barnsley, J. Elton, D.P. Hardin, P.R. Massopust, Hidden variable fractal interpolation functions, SIAM J. Math. Anal. 20(5) (1989) 1218-1248.

\bibitem{AC1} A. Carvalho, Box dimension, oscillation and smoothness in function spaces. J. Funct. Spaces Appl. 3(3) (2005) 287-320.
\bibitem{Cazassa} P. G. Cazassa, O. Christensen, Perturbation of operators and application to
frame theory, J. Fourier Anal. Appl. 3(5) (1997) 
543-557.

\bibitem{SS2} S. Chandra, S. Abbas, The calculus of bivariate fractal interpolation surfaces, Fractals 29(3) (2021) 2150066.
\bibitem{SS3} S. Chandra, S. Abbas, On fractal dimensions of fractal functions using function spaces, Bull. Aust. Math. Soc. 106(3) (2022) 470-480.
\bibitem{SS} S. Chandra, S. Abbas, On fractal dimension of the graph of non-stationary fractal interpolation function, Contemporary Mathematics-American Mathematical Society 797 (2024).

\bibitem{B} D. Celik, S. Kocak, Y. \"Ozdemir, Fractal interpolation on the Sierpi\'nski gasket, J. Math. Anal. Appl., 337 (2008) 343-347.
\bibitem{DS2} K. Dalrymple, R. S. Strichartz, J. P. Vinson, Fractal differential equations on the Sierpinski gasket, Journal of Fourier analysis and Applications, 5(2) (1999) 203-284.
	\bibitem{ADBJ} A. Deliu, B. Jawerth, Geometrical dimension versus smoothness, Constr. Approx. 8 (1992) 211-222.
    \bibitem{GP1} D. Gupta, A. Pandey, Analyzing impact of corporate governance index on working capital management through fractal functions,
Chaos, Solitons \& Fractals, 183 (2024), Article 114946.
\bibitem{DS1} Z. Douzi, B. Selmi, Multifractal variation for projections of measures, Chaos, Solitons \& Fractals 91 (2016) 414-420.
	\bibitem{Fal} K. J. Falconer, Fractal Geometry: Mathematical Foundations and Applications, John Wiley Sons Inc., New York, 1999.
		\bibitem{Fraser} K. J. Falconer, J. M. Fraser, The horizon problem for prevalent surfaces, Mathematical Proceedings of the Cambridge Philosophical Society 151(2) (2011) 355-372.
\bibitem{GVS} Gurubachan, V. V. M. S. Chandramouli, S. Verma, Fractal Dimension of $\alpha$-Fractal Functions Without Endpoint Conditions, Mediterr. J. Math. 21 (2024), no. 3, Paper No. 71, 23 pp.
\bibitem{GVS2} Gurubachan, V. V. M. S. Chandramouli, S. Verma, Analysis of $\alpha$-fractal functions without boundary point conditions on the Sierpi\'nski gasket, Applied Mathematics and Computation 486, (2025) 129072.

\bibitem{GVS3} Gurubachan, V. V. M. S. Chandramouli, S. Verma, Constrained approximation and set-valued results for a class of fractal functions, Discrete and Continuous Dynamical Systems-S (2025) 1-25, doi: 10.3934/dcdss.2025096.

\bibitem{HP1} D. P. Hardin, P. R. Massopust, The capacity for a class of fractal functions, Commun. Math. Phys. 105 (1986) 455-460.
\bibitem{HP2} D. P. Hardin,  P. R. Massopust, Fractal interpolation functions from to and their projections, Zeitschrift für Analysis u. i. Anw. 12 (1993) 535-548. 
	\bibitem{H} J. E. Hutchinson, Fractals and self-similarity, Indiana Uni. Math. J. 30(5) (1981) 713-747.
	\bibitem{J1} H. Jebali, Fractal Functions and Fractal Dimensions on the Product of the Sierpi\'nski Gaskets, Results Math 80 (2025) 110.
    \bibitem{JV1} S. Jha, S. Verma, Dimensional Analysis of $\alpha$-fractal Functions, Results in Mathematics 76(4) (2021) 186.
	\bibitem {Kigami} J. Kigami, Analysis on Fractals, Cambridge University Press, Cambridge, UK, 2001.

    \bibitem{KBV11} A. Kumar, S. Boulaaras, S. K. Verma and M. Biomy, Non-stationary fractal functions on the Sierpi\'nski gasket, Mathematics, 12 (2024), 3463. doi: 10.3390/math12223463.

 \bibitem{KBV12} A. Kumar, S. K. Verma, and S. Boulaaras, On $\alpha$-fractal functions and their applications to analyzing the S \& P BSE Sensex in India, Chaos, Solitons \& Fractals, 186 (2024), Paper No. 115194, 8 pp.
\bibitem{RSP} R. Lal, S. Chandra, A. Prajapati, Fractal surfaces in Lebesgue spaces with respect to fractal measures and associated fractal operators, Chaos, Solitons \& Fractals 181 (2024), 114684, 7 pp.

\bibitem{RBS} R. Lal, B. Selmi and S. Verma, On dimension of fractal functions on product of the Sierpi\'nski gaskets and associated measures, Results in Mathematics 79 (2) (2024) 73.


\bibitem{PR1}  P. R. Massopust,  Fractal Functions, Fractal Surfaces, and Wavelets, 2nd edn. Academic Press, San Diego, 2016.
\bibitem{MP1} M. Megala, S. A. Prasad, Spectrum of a Self-Affine Measure with Four-Element Digit Set, Fractals 30 (04), (2022) 2250087.
\bibitem{MP2} M. Megala, S. A. Prasad, Spectrum of self-affine measures on the Sierpinski family, Monatshefte f\"Ur Mathematik 204(1) (2024) 157-169.
\bibitem{MP3} M. Megala, S. A. Prasad, Spectrality of certain self-affine measures, Contemporary Mathematics (AMS), Vol. 825, (2025) 1-10.
\bibitem{MNS1} R. N. Mohapatra, M. A. Navascu\'es, M.V. Sebasti\'an, Iteration of operators with contractive mutual relations of Kannan type, Mathematics 10(15) (2022) 2632.
\bibitem{Moran} M. Mor\'an, J. M. Rey, Singularity of self-similar measures with respect to Hausdorff measures, Transactions of
the American Mathematical Society 350(6) (1998) 2297-2310.
\bibitem{NV1} M. A. Navascu\'es, Fractal polynomial interpolation, Zeitschrift für Analysis und ihre Anwendungen 24(2) (2005)  401-418.
\bibitem{NV2} M. A. Navascu\'es, Fractal approximation, Complex Analysis and Operator Theory 4 (2010) 953-974.
\bibitem{NV3}  M. A. Navascu\'es, Fractal bases of Lp spaces, Fractals 20(02) (2012) 141-148.
\bibitem{NVV2}  M. A. Navascu\'es, S. Verma, P. Viswanathan, Concerning the vector-valued fractal interpolation functions on the Sierpi\'nski gasket, Mediterr. J. Math., 18 (2021) article no. 202.
\bibitem{Mnew2} M. A. Navascues, S. Verma, Non-stationary $\alpha$-fractal surfaces, Mediterr. J. Math., 20(1) (2023) 48.
\bibitem{PV1} S. A. Prasad, S. Verma, Fractal interpolation function on products of the Sierpi\'nski gaskets, Chaos, Solitons \& Fractals 166 (2023) 112988.
\bibitem{Ri2} S.-I. Ri, Fractal Functions on the Sierpi\'nski Gasket, Chaos, Solitons and Fractals 138 (2020) 110142.

\bibitem {Ruan4} S.-G. Ri, H.-J. Ruan, Some properties of fractal interpolation functions on Sierpi\'nski gasket, J. Math. Anal. Appl. 380 (2011) 313-322.
\bibitem {Ruan3} H.-J. Ruan, Fractal interpolation functions on post critically finite self-similar sets, Fractals 18 (2010) 119-125.
\bibitem{RSY} H-J. Ruan, W-Y. Sub, K. Yao, Box dimension and fractional integral of linear fractal interpolation functions, J. Approx. Theory 161 (2009) 187-197.
\bibitem{SP} A. Sahu, A. Priyadarshi, On the box-counting dimension of graphs of harmonic functions on the Sierpi\'{n}ski gasket, J. Math. Anal. Appl. 487 (2020) 124036. 	
\bibitem{BS1} B. Selmi, General multifractal dimensions of measures, Fuzzy Sets and Systems, (2025) 109177.
\bibitem{BS2} B. Selmi, Some new characterizations of Olsen’s multifractal functions, Results in Mathematics 75(4) (2020) 147.
\bibitem{RS} R. S. Strichartz, Differential Equations on Fractals, Princeton University Press, Princeton, NJ, 2006.
\bibitem {D} H. J. Ruan, Fractal interpolation functions on post-critically finite self-similar sets, Fractals 18 (2010) 119-125.

\bibitem{MV1} M. Verma, Dimensional approximation of inhomogeneous attractor without any separation condition, Bulletin of the Australian Mathematical Society (2025) 1-13, doi: 10.1017/S0004972725100129.
\bibitem{Mverma1} M. Verma, A. Priyadarshi, S. Verma, Analytical and dimensional properties of fractal interpolation functions on the Sierpiński gasket, Fract. Calc. Appl. Anal. 26 (2023) 1294–1325.
\bibitem{Mverma2} M. Verma, A. Priyadarshi, Dimensions of new fractal functions and associated measures, Numer. Algorithms (2023) 1-30.

\bibitem{Mverma3} M. Verma, A. Priyadarshi, New type of fractal functions for the general data sets, Acta
Applicandae Mathematicae, 187(12) (2023) 1-23.
\bibitem{VM} S. Verma, P. R. Massopust, Dimension preserving approximation, Aequationes Mathematicae 96(6) (2022) 1233-1247.
\bibitem{VS1} S. Verma, A. Sahu, Bounded Variation on the Sierpi\'nski Gasket, Fractals 30 (07) (2022) 1-12.
\bibitem{VJN1} S. Verma, S. Jha, M. A. Navascu\'es, Smoothness analysis and approximation aspects of non-stationary bivariate fractal functions, Chaos, Solitons \& Fractals 175 (2023) 114003.
 \bibitem{VAM} S. Verma, E. Agrawal, M. Megala, Some remarks on the exponential separation and dimension preserving approximation for sets and measures, (2025), 1-19, arXiv preprint, doi.org/10.48550/arXiv.2508.07692.
\bibitem{VP1} S. Verma, A. Priyadarshi, Further analysis of Hausdorff dimension and separation conditions, Contemporary Mathematics (AMS), Vol. 825, (2025) 1-14, doi.org/10.1090/conm/825/16517.
\bibitem{VAD} S. Verma, E. Agrawal, S. Dubey, Some results on Lower Assouad and quantization dimensions, arxiv preprint (2025) submitted for publication.
\bibitem{YSL} B. Yu, B. Selmi, Y. Liang, General fractal dimensions of graphs of continuous functions associated with the Katugampola Fractional Integral, Fractals 33(05) (2025) 2550040.
\bibitem{YLC} B. Yu, Y. Liang, S. Chandra, On the Dimensional Invariance of Graphs of Fractal Functions Under Horizontal Nonlinear Variation, Mathematical Methods in the Applied Sciences (2025).
\end{thebibliography}

\end{document}